\documentclass
{amsart}

\usepackage{amssymb,amscd}
\usepackage[all,line,arc,curve,color,frame,pdf]{xy}
\usepackage{tikz}
\usetikzlibrary{positioning, trees, snakes}
\usepackage[textsize=tiny]{todonotes}
\usepackage{enumitem}
\usepackage[symbol]{footmisc}
\numberwithin{equation}{subsection}
\theoremstyle{plain}
\newtheorem{thm}{Theorem}[section]
\newtheorem{cor}[thm]{Corollary}
\newtheorem{lemma}[thm]{Lemma}
\newtheorem{prop}[thm]{Proposition}
\newtheorem{assumpt}[thm]{Assumptions}

\theoremstyle{remark}
\newtheorem{rem}[thm]{Remark}
\newtheorem{ex}[thm]{Example}

\theoremstyle{definition}
\newtheorem{defi}[thm]{Definition}

\usepackage[textsize=tiny]{todonotes}

\newcommand\ignore[1]{}

\DeclareMathOperator{\Aut}{Aut}

\DeclareMathOperator{\Hom}{Hom}
\DeclareMathOperator{\HH}{H}

\DeclareMathOperator{\rank}{rank}
\DeclareMathOperator{\rk}{rk}
\DeclareMathOperator{\conv}{conv}
\DeclareMathOperator{\Pic}{Pic}
\DeclareMathOperator{\Proj}{Proj}

\newcommand\CC{{\mathbb{C}}}
\newcommand\PP{{\mathbb{P}}}
\newcommand\RR{{\mathbb{R}}}
\newcommand\QQ{{\mathbb{Q}}}
\newcommand\ZZ{{\mathbb{Z}}}

\newcommand\NN{{\bf{N}}}

\newcommand\cY{{\mathcal Y}}
\newcommand\cO{{\mathcal O}}
\newcommand\cV{{\mathcal V}}

\newcommand\cN{{\mathcal N}}

\newcommand\cC{{\mathcal C}}

\newcommand\cR{{\mathcal R}}
\newcommand\cE{{\mathcal E}}

\newcommand\cA{{\mathcal A}}
\newcommand\cG{{\mathcal G}}
\newcommand\cL{{\mathcal L}}
\newcommand\cS{{\mathcal S}}

\newcommand\ra{{\ \rightarrow\ }}

\newcommand\iso{{\ \cong\ }}

\begin{document}
\title{Adjunction for varieties with a $\CC^*$ action}
\author[Romano]{Eleonora A. Romano}
\address{Institute of Mathematics,
University of Warsaw, 
PL-02-097 Warszawa
}
\email{elrom@mimuw.edu.pl, J.Wisniewski@uw.edu.pl}

\author[Wi{\'s}niewski]{Jaros{\l}aw A. Wi{\'s}niewski}

\begin{abstract}
  Let $X$ be a complex projective manifold, $L$ an ample line bundle
  on $X$, and assume that we have a $\CC^*$ action on $(X,L)$. We
  classify such triples $(X,L,\CC^*)$ for which the closure of a general
  orbit of the $\CC^*$ action is of degree $\leq 3$ with
  respect to $L$ and, in addition, the source and the sink of the
  action are isolated fixed points, and the $\CC^*$ action on the normal
  bundle of every fixed point component has weights $\pm 1$. We treat
  this situation by relating it to the classical adjunction theory. As
  an application, we prove that contact Fano manifolds of dimension 11
  and 13 are homogeneous if their group of automorphisms is reductive
  of rank $\geq 2$.
\end{abstract}

\thanks{The project has been supported by Polish National Science
  Center grants 2013/08/A/ST1/00804 and 2016/23/G/ST1/04282. Thanks to Joachim Jelisiejew, Gianluca Occhetta,
  {\L}ukasz Sienkiewicz, Luis Sol{\'a} Conde, Andrzej Weber for
  discussions; and to the referees for their valuable comments which improved the exposition of the present paper.}
\maketitle
\tableofcontents

\setcounter{section}{-1}
\section{Introduction}
\subsection{A view on manifolds with a $\CC^*$ action}
Let us recall that Amplitude Modulation (AM) and Frequency Modulation (FM) are
two different technologies of broadcasting radio signals. AM works by modulating
the amplitude of the signal with constant frequency. In FM technology the
information is encoded by varying the frequency of the wave with amplitude being
constant. In the present paper we adopt the idea of
  passing the information via either AM or FM technology to deal with
  varieties with a $\CC^*$ action.

Given a complex projective variety $X$ with an ample line bundle $L$
and an action of $\CC^*$ on $(X,L)$, we can study this set up
in two ways: (1) by examining the amplitude of $L$ on curves on $X$ (AM
technology) and (2) by understanding the weights of a linearization of the
action of $\CC^*$ on $L$ over the connected components of the fixed point set
of this action (FM technology).

The structure of $X$ with a $\CC^*$ action can be encoded in a graph whose
vertices are components of the fixed point locus, and the edges are orbits
whose closures meet the respective components. Given a linearization $\mu_L$ of
the line bundle $L$, to each component of the fixed point locus one can associate
the weight in $\Hom(\CC^*,\CC^*)=\ZZ$ with which $\CC^*$ acts on fibers of
$L$ over the component in question. Now, the radio analogy goes as follows: one
can relate the values of $\mu_L$ (frequencies of $L$) on the components of the fixed
point set of the $\CC^*$ action with the degree (the volume) of $L$ on the
closures of orbits joining respective fixed point set components.

Namely, given a $\CC^*$ equivariant morphism $f:\PP^1\rightarrow X$ we get the
following identity (see Lemma \ref{AMvsFM}):
\begin{equation}\tag{AM$\leftrightarrow$FM} \label{eqAMvsFM}
\delta\cdot\deg f^*L=\mu_L(f(0))-\mu_L(f(\infty))
\end{equation}
where $0,\infty$ are the fixed points of the action of $\CC^*$ on $\PP^1$, and
$\delta=\delta(T_0\PP^1)$ is the weight of the $\CC^*$ action on the tangent of
$\PP^1$ at $0$. Thus, the left hand side of the above equality measures
the amplitude of the line bundle $L$, while the right hand side measures the
difference of the weights of the $\CC^*$ action on the fibers of $f^*L$ over
the fixed points. 
In view of the (\ref{eqAMvsFM}) equality we define the \textit{bandwidth} of a pair $(X,L)$ as the degree of the closure a general orbit of the $\CC^*$ action with respect to $L$, and we are interested in classifying some pairs $(X,L)$ admitting a $\CC^*$ action of small bandwidth. 

In \cite{FUJITA,IONESCU} Ionescu and Fujita proved classification results for polarized pairs $(X,L)$ by looking at the \textit{nef value} $\tau=\tau(X,L):=\min\{t\in \mathbb{R}: K_X+tL \ \ {\rm is\  nef}\}$ (see Theorem \ref{ionescu}).
In this paper, assuming that we have a nontrivial $\CC^*$ action on $(X,L)$, we will make use of the (\ref{eqAMvsFM}) equality to study the positivity of the divisor $K_X+tL$, so that we are able to compute the nef value of $(X,L)$ or find an estimate of it. Combining this information with classical results from adjunction theory, we obtain a first classification result for bandwidth one and two varieties (see Theorem \ref{bwleq2classification}). As a main application of our approach we study pairs $(X,L)$ of bandwidth three which emerged naturally in the context of the LeBrun-Salamon conjecture (see Theorem \ref{bw3classification}). To this end, the technique consists again in relating new methods and properties due to the torus actions arising from Bia{\l}ynicki-Birula decomposition (c.f. Theorem \ref{thm_ABB-decomposition}) with the more classical adjunction machinery.  

\subsection{Motivation and contents of the paper}
The celebrated Le\-Brun-Salamon conjecture in Riemannian geometry asserts that
the only positive quaternion-K{\"a}hler manifolds are Wolf spaces. Its
algebro-geometric counter-part asserts that the closed orbits in
projectivizations of adjoint representations of simple algebraic group are the
only Fano contact manifolds. Recently, in \cite{B_W} the combinatorics of torus
action has been used to prove the conjecture in low dimensions. In the present
paper we use the techniques of a $\CC^*$ action on pairs $(X,L)$ as above,  to
prove the following extension of previous results, see also Theorem
\ref{contact_dimleq13} for a more detailed formulation. \par\bigskip
\noindent {\bf Theorem.} {\it Let $X_\sigma$ be a Fano contact manifold of dimension
$\leq 13$ and $\Pic{X_\sigma}=\ZZ L_{\sigma}$. If the group of contact automorphisms $G$ is reductive of rank
$\geq 2$ then $X_\sigma$ is the closed orbit in the projectivization of the adjoint
representation of a simple algebraic group.}
\par\bigskip
It is known that the contact manifold coming from a
quaternion-K\"ahler manifold admits K\"ahler-Einstein metric, so that when dealing with LeBrun-Salamon conjecture the varieties in question have the group of the 
contact automorphisms reductive, hence this assumption on $G$ is not restrictive (see \cite{SALAMON}). Moreover, earlier results were for contact Fano manifolds $X$ with $\dim
X\leq 9$ and without lower bound on the rank of the group of its automorphisms. Notice that, being the dimension of the Lie algebra of $G$ equal to $h^0(X,L)$ (see for instance \cite[Lemma 4.5]{B_W}), then  the assumption on the rank is true if e.g. $h^0(X,L)>3$.
We refer to
Section \ref{sect-contact-mnflds}, where after recalling past and recent results in the context of 
the LeBrun-Salamon conjecture, we apply new methods from adjunction theory for varieties with a $\CC^*$ action to solve the conjecture under the assumptions of the above theorem. 

Indeed, following the strategy of \cite{B_W}, to deal with LeBrun-Salamon
conjecture we need to classify polarized pairs $(X,L)$ of small 
  bandwidth. 
 As will be explained in Subsection \ref{$SL_3$ action on contact manifolds}, such pairs $(X,L)$ appear in our analysis as subvarieties of the initial Fano contact manifold, and we need to study them to collect all the combinatorial data of the action as a crucial step to show the above theorem. To this end, we use tools and new methods developed in the previous sections, concerning adjunction theory for varieties admitting a $\CC^*$ action. In this framework, the main
technical result of the paper is Theorem \ref{bw3classification}
describing polarized pairs $(X,L)$ with an action of $\CC^*$ of bandwidth
three which satisfies some technical assumptions that are natural for
the application to contact manifolds. Denoting by $n$ the dimension of $X$ with $n\geq 3$, the result is the following list of
possibilities:
\begin{enumerate}[leftmargin=*]
  \item $(X,L)=(\PP(\cV),\cO(1))$ is a scroll over $\PP^1$,
    where $\cV$ is either $\cO(1)^{n-1}\oplus\cO(3)$ or
    $\cO(1)^{n-2}\oplus\cO(2)^2$, or
  \item $(X,L)$ is a quadric bundle $(\PP^1\times\mathcal{Q}^{n-1}, \cO(1,1))$, or
  \item $n\geq 6$ is divisible by 3 and $X$ is Fano, $\rho_X=1$, $-K_X=\frac{2}{3}nL$.
\end{enumerate}

In order to obtain the above classification, in Section
\ref{section-adjunction} we relate the classical adjunction theory
(see \cite{B_S, FUJITA-Sendai, IONESCU}) and Mori theory (see \cite{KMM,
  KOLLAR_MORI}) with a combinatorial description of a manifold with a
$\CC^*$ action.  In fact, types (1) and (2) of pairs $(X,L)$ in the above list are described in terms of their adjunction morphism. Type (2)
in the above list leads to contact manifolds which are homogeneous with
respect to $SO$ groups, as described in the Appendix of the present
paper. Type (3) in the same list have been recently classified in  \cite{OSCRW2} by using different methods from birational and projective geometry. In total, there are four of these varieties, all of them are rational homogeneous; we refer to \cite[Theorem 6.8]{OSCRW2} for their complete list. 
In the recent preprint \cite{OSCRW} the varieties of type (3) are related to contact manifolds homogeneous with respect to four exceptional simple groups of $F_4$, $E_6$, $E_7$ and $E_8$ type. Dealing also with such cases in which $G$ is of exceptional type, in \cite[Theorem 6.1]{OSCRW2} LeBrun-Salamon conjecture has been proved in arbitrary dimension, under certain assumptions on the rank of the maximal torus. 
\subsection{Notation}
The following notation is used throughout the article.
\begin{itemize}[leftmargin=*]
\item $X$ is a complex projective normal variety of dimension $n$. For the most
part of the paper we assume $X$ smooth with an ample line bundle $L$, so that
$(X,L)$ is a \textit{polarized pair}.
\item Given a polarized pair $(X,L)$ we denote by $\tau=\tau(X,L)$ the
\textit{nef value}, namely $\tau(X,L):=\min\{t \in\RR: K_X+tL \ \ {\rm is\ nef}\}$,
moreover $\phi_{\tau}:=\phi_{K_X+\tau L}\colon X\rightarrow X'$ is the
\textit{adjunction} (or \textit{adjoint}) \textit{morphism}.
\item We denote by $H=(\CC^*)^r$ an algebraic torus of arbitrary rank $r$, acting on $X$.  Moreover, we denote by $M=\Hom_{alg}(H,\CC^*)\iso\ZZ^r$ the set
of characters (or weights) of $H$. 
\item $X^{H}=\bigsqcup_{i\in I} Y_i$ is the fixed locus of the $H$ action, where $I$
is a set indexing its connected components; by $\cY=\{Y_i\}$ we denote the set
of the irreducible fixed point components of $X^{H}$.
\item For an arbitrary line bundle $\cL \in \Pic X$ we denote by $\mu_\cL\colon H\times \cL\rightarrow \cL$ (or simply by $\mu$) a linearization of the action
of $H$ on $\cL$. By abuse, we continue to denote by $\mu_\cL\colon \cY\ra M\iso\ZZ^r$ the associated map on the set of fixed point components, which we call
\textit{fixed point weight map}, see Definition \ref{mu-map}.
\item Given a $\CC^*$ action on $X$, and a nef line bundle $\cL \in \Pic{X}$
admitting a linearization $\mu=\mu_\cL$, the \textit{bandwidth} of the
triple $(X,\cL,\CC^*)$ is defined as $|\mu|=\mu_{\max}-\mu_{\min}$ where
$\mu_{\max}$ and $\mu_{\min}$ denote the maximal and minimal value of the
function $\mu_\cL$, see Definition \ref{bandwidth}.
\end{itemize}

\section{Preliminaries}
In the present section we recall basic definitions and properties of
adjunction and Mori theory as well as regarding varieties with a $\CC^*$
action. We refer the reader to \cite{KOLLAR_MORI} for a detailed
exposition on Mori theory, and to \cite{B_S, FUJITA-Sendai, IONESCU} for an account
on adjunction theory.
We work over the field of complex numbers, with projective, irreducible, reduced varieties.

\subsection{Adjunction and Mori theory}%
\label{subsection-prelim-adjunct}
Let $X$ be a normal projective variety of arbitrary dimension $n$.
Let us denote by $\NN^1(X)$ (respectively $\NN_1(X)$) the $\RR$-spaces
of Cartier divisors (respectively, 1-cycles on $X$), modulo numerical
equivalence.  We denote by $\rho_{X}:=
\dim{\NN_{1}(X)}=\dim{\NN^{1}(X)}$ the \textit{Picard
  number} of $X$, and by $[\cdot]$ the numerical equivalence class in
$\NN_1(X)$, and in $\NN^1(X)$.
The intersection of divisors and curves
determines a nondegenerate bilinear pairing of these two
$\RR$-spaces. We consider cones $\cC(X)\subset\NN_1(X)$ and
$\cA(X)\subset\NN^1(X)$ spanned by classes of effective curves and
classes of ample divisors, respectively. Their closures (in the standard
topology on $\RR$-spaces) are dual in terms of the intersection
product.

A \textit{contraction} of $X$ is a surjective morphism with connected
fibers $\phi\colon X\to Y$ onto a normal projective variety. Any
contraction yields a surjective linear map $\phi_*\colon \NN_1(X)\to
\NN_1(Y)$ given by the push-forward of 1-cycles, and the pull-back of
Cartier divisors $\phi^*\colon \NN^1(Y)\to\NN^1(X)$ such that
$\phi^*([D])=[\phi^*(D)]$.

The case of our main interest is the following situation.
\begin{assumpt}\label{gen-assumpt}
  Let $(X,L)$ be a polarized manifold, namely $X$ is a smooth
  projective variety of dimension $n$ and $L$ is an ample line bundle
  on it. In addition we assume that the variety $X$ admits a nontrivial $\CC^*$ action, that is $\CC^*\times X\rightarrow X$, with
  a linearization $\mu\colon \CC^*\times L\rightarrow L$.%
  \footnote[3]{Note that in Section \ref{sect-contact-mnflds} we
    consider the case when the variety is a contact manifold
    $X_\sigma$ of dimension $2n+1$ with an action of a torus
    $\widehat{H}$ of rank $\geq 2$.  }
\end{assumpt}

For the polarized pair $(X,L)$ we define its {\em nef value} as
follows:
$$\tau=\tau(X,L):=\min\{t\in \mathbb{R}: K_X+tL \ \ {\rm is\  nef}\}.$$
We note that if $X$ admits a $\CC^*$ action then it is uniruled, hence $K_X$ is
not nef, so that $\tau>0$. Thus, by Kawamata rationality theorem (see
\cite[Theorem 4.1.1]{KMM}) one has $\tau \in \QQ$. Moreover, Kawamata-Shokurov
Base Point free Theorem provides the adjunction morphism
$$\phi_{\tau}:=\phi_{K_X+\tau L}: X\rightarrow X'$$ such that $K_X+\tau
L=\phi_{\tau}^* L'$ for some $\QQ$-Cartier ample divisor $L'$ on $X'$. The
variety $X'$ is normal and $\phi_{\tau}$ has connected fibers, namely
$\phi_{\tau}$ is a contraction of $X$. In fact
\begin{equation}\label{proj}
X'=\Proj\left(\bigoplus_{m\geq 0}\HH^0(X,m(K_X+\tau
  L))\right)
\end{equation}
where $m$ is such that $m(K_X+\tau L)$ is Cartier.

The following result is due to Ionescu and Fujita (see \cite{IONESCU} and also
\cite{FUJITA-Sendai}), and will be crucial for proving the results in
Sections \ref{section-bandwidth} and \ref{Class_bandwidth3}.

\begin{thm}\label{ionescu}
  Let $(X,L)$ be a polarized pair. Then
  $\tau\leq n+1$ with equality only for projective space, that is if
  $(X,L)=(\PP^n,\cO(1))$.
    \begin{enumerate}[leftmargin=-2pt]
    \item Suppose that $n\geq 2$ and $\tau<n+1$. Then $\tau\leq n$ with
      equality only if
    \begin{enumerate}
     \item either $X$ is a smooth quadric, that is $(X,L)=(\mathcal{Q}^n,\cO(1))$, or
     \item $(X,L)$ is a $\PP^{n-1}$-bundle over a smooth curve with
       $L$ relative $\cO(1)$.
    \end{enumerate}

  \item Suppose that $n\geq 3$ and $\tau<n$. Then $\tau\leq n-1$ with
    equality only if one of the following holds:
    \begin{enumerate}
    \item $(X,L)$ is a del Pezzo manifold, that is $-K_X=(n-1)L$; see
      \cite{FUJITA, ISK} for their complete classification.
    \item $(X,L)$ is a quadric bundle over a smooth curve with $L$
      relative $\cO(1)$.
    \item $(X,L)$ is a $\PP^{n-2}$-bundle over a smooth surface with
      $L$ relative $\cO(1)$.
    \item The adjoint morphism $\phi_{n-1}\colon X\rightarrow X'$ is a
      birational morphism contracting a finite number of disjoint
      divisors $E_i\iso\PP^{n-1}$ to smooth points of $X'$ and
      $L_{|E_i}\iso\cO(1)$; there exists an ample line bundle $L'$
      over $X'$ such that $\phi_{n-1}^*L'=K_X+(n-1)L$\ignore{, and $X'$ does
      not contain divisors $\tilde{E}_i\iso\PP^{n-1}$ such that
      $L'_{|\tilde{E}_i}\iso\cO(1)$. The polarized pair $(X',L')$ is
      called a reduction of $(X,L)$}.
    \end{enumerate}
    \end{enumerate}
\end{thm}
The following observation follows easily by taking a rational curve
$C\subset X$ which spans an extremal ray contained in the extremal face contracted by the adjoint
morphism $\phi_{\tau}\colon X\to X^{\prime}$, and using that
$\tau=\frac{-K_X\cdot C}{L\cdot C}\leq \frac{n+1}{L\cdot C}$.
\begin{rem}\label{tau_integer} Let $(X,L)$ be a polarized pair with
  $n\geq 3$. Assume that $\tau>n-2$. Then $\tau\geq n-1$, and $\tau\in
  \ZZ$ except for $(X,L)=(\PP^{4}, \cO(2))$, $(X,L)=(\PP^{3},
  \cO(3))$, and $(X,L)=(\mathcal{Q}^{3}, \cO(2))$.
\end{rem}


\subsection{Varieties with a $\CC^*$ action}
Let us consider an effective (i.e.~nontrivial) action of an algebraic
torus $H=(\CC^*)^r$ on a smooth projective variety $X$, that is
$H\times X\ni (t,x)\rightarrow t\cdot x\in X$. 
Given a subtorus $H'\subseteq H$ we can consider the resulting action $H'\times X\rightarrow X$;
this operation will be called \textit{downgrading} the action of $H$ to
$H'$ (see \cite[$\S$2.2]{B_W} for further details). The action is called \textit{almost faithful}
if the resulting homomorphism $H\rightarrow Aut(X)$ has finite kernel.

Except for Section \ref{sect-contact-mnflds}, we will be primarily interested
in the case $r=1$.

We consider the fixed locus of the action $X^H$ and its decomposition
into connected components: $$X^H=\bigsqcup_{i\in I}Y_i$$ where $I$ is a
set of indices, and each component $Y_i$ is a smooth subvariety (see
e.g.~the main theorem in \cite{IVERSEN}).  By $\cY=\{Y_i: i\in I\}$ we
denote the set of the irreducible fixed point components of $X^H$.
\begin{rem}\label{irreducile_vs_connected}
  We stress that if $X$ is smooth then the connected components of $X^H$
  are smooth, hence irreducible. If $X$ is not smooth then the connected
  components of $X^H$ may not be irreducible as the following example
  shows (thanks to Joachim Jelisiejew): consider the quadric cone
  $\mathcal{Q}=\{z_1z_2+z_2z_4=0\}$ in the projective space with coordinates
  $[z_0,z_1,\dots,z_4]$ and a $\CC^*$ action with weights
  $(0,0,0,1,-1)$. Then the fixed point set consists of two isolated
  points $[0,0,0,1,0]$, $[0,0,0,0,1]$, and the reducible conic
  $\mathcal{Q}\cap\{z_3=z_4=0\}$. \end{rem}

We have the following standard observation.
\begin{lemma}\label{curve-cone-generation}
  Let $X$ be a variety with an effective $\CC^*$ action. Then the
  cone of curves $\cC(X)$ is generated by classes of closures of
  orbits and by classes of curves contained in the fixed
  locus of the action.
\end{lemma}
\begin{proof}
  The result follows by applying standard Mori breaking technique
  using the $\CC^*$ action, see e.g. ~\cite[p.~253]{W_NOTES} for
  details. Let us take an arbitrary irreducible curve $C\subset X$
  with normalization $f:\widehat{C}\rightarrow C\subset X$. We
  consider the morphism $F: \CC^*\times\widehat{C}\rightarrow X$
  defined by setting
  $$\CC^*\times\widehat{C}\ni(t,p)\mapsto F(t,p)=t\cdot f(p)\in X.$$
  We extend the morphism $F$ to a rational map $\CC\times
  \widehat{C}\dashrightarrow X$ which we resolve to a regular morphism
  $\widehat{F}: \widehat{S}\rightarrow X$ blowing-up the product over
  $0=\CC\setminus\CC^*$.  The image (as a 1-cycle) under $\widehat{F}$
  of the fiber of $\widehat{S}\rightarrow\CC\times
  \widehat{C}$ over 0 is the sum of
  curves which are stable under the $\CC^*$ action and it is
  numerically equivalent to $C$.
\end{proof}

For every $Y\in \cY$ the torus $H$ acts on $TX_{\mid Y}$ so that
we get the decomposition $TX_{\mid Y} =T^{+}\oplus T^{0}\oplus T^{-}$, where $T^{+}$, $T^{0}$, $T^{-}$ are respectively the
subbundles of $TX_{\mid Y}$ on which $H$ acts with positive,
zero or negative weights. Then, by local linearization, $T^0=TY$ and
$$T^{+}\oplus T^{-}=\cN_{Y/X}=\cN^+(Y)\oplus \cN^-(Y)$$
is the decomposition of the normal bundle $\cN_{Y/X}$ into the part on
which $H$ acts with positive, respectively, negative weights.

\begin{defi}\label{def-equalized}
  Setting as above. We say that the $\CC^*$ action on $X$ is {\em
    equalized} if for every component $Y\in\cY$ the torus acts
  on $\cN^+(Y)$ with all the weights equal to $+1$ and on $\cN^-(Y)$ with all the weights equal to $-1$.
\end{defi}

It is a basic fact (see \cite{SOM}) that for $x\in X$ the action $\CC^*\times
\{x\}\to X$ extends to a holomorphic map $\PP^1\times \{x\}\to X$, hence there
exist $\lim_{t\rightarrow 0}t\cdot x$, and $\lim_{t\rightarrow \infty}t\cdot x$.
Moreover, since the orbits are locally closed, and the closure of an orbit is an
invariant subset, then both the limit points of an orbit lie in $\cY$. We will
call these limits the {\em source} and the {\em sink} of the orbit of $x$,
respectively.

For every $Y\in\cY$ we can define the Bia{\l}ynicki-Birula cells in the
following way:
$$X^+(Y)=\{x\in X: \lim_{t\rightarrow 0} t\cdot x\in Y\}{\rm \ \ and \ \
}X^-(Y)=\{x\in X: \lim_{t\rightarrow \infty} t\cdot x\in Y\}.$$
The following result is due to Bia{\l}ynicki-Birula and known as BB
decomposition. We use this argument as presented in \cite{CARRELL}. See
\cite{BB} for the original exposition. A vast generalization of this result,
which is also valid for singular varieties, can be found in a recent paper
\cite{J-S} and references therein.
\begin{thm}\label{thm_ABB-decomposition}
In the situation described above the following holds:
  \begin{itemize}[leftmargin=*]
  \item $X^{\pm}_i$ are locally closed subsets and there are two decompositions
  $$X=\bigsqcup_{i\in I}X^+(Y_i)=\bigsqcup_{i\in I}X^-(Y_i)$$ which we call
  $X^+$ or $X^-$ BB decomposition, respectively.
\item For every $Y\in\cY$ there are $\CC^*$-isomorphisms $X^+(Y)\cong
  \cN^{+}(Y)$ and $X^-(Y)\cong \cN^{-}(Y)$ lifting the natural maps
  $X^\pm(Y)\rightarrow Y$. Moreover, the map $X^\pm(Y)\rightarrow Y$ is algebraic and is a
  $\CC^{\rk^\pm(Y)}$ fibration, where we set $\rk^\pm(Y):=\rank \cN^\pm(Y)$.
\item There is a decomposition in homology
  $$H_m(X,\ZZ)=\bigoplus_{i\in I}H_{m-2\rk^+(Y_i)}(Y_i,\ZZ)=
  \bigoplus_{i\in I}H_{m-2\rk^-(Y_i)}(Y_i,\ZZ).$$
\end{itemize}
\end{thm}
The unique $Y$ such that $X^+(Y)$ is dense in $X$ is called the
\textit{source} of the action. The unique $Y$ such that $X^-(Y)$ is
dense in $X$ is called the \textit{sink}.

We have a partial order on $\cY$ in the following
way:
\begin{equation} \label{partial_order}
Y_i\prec Y_j\Leftrightarrow \exists \ x\in X:\
\lim_{t\rightarrow 0} t\cdot x\in Y_i{\rm \ \ and \ \
}\lim_{t\rightarrow \infty} t\cdot x\in Y_j
\end{equation}

\begin{defi}\label{onepointend-def}
  An effective $\CC^*$ action on a smooth variety $X$ is said to
  have \textit{one pointed end} if its source or sink is a single
  point. The action is said to have \textit{two pointed ends} if both
  the source and the sink are isolated points.
\end{defi}

We note that replacing $t$ with $t^{-1}$ we change the action to the
opposite and $X^+$ decomposition into $X^-$ decomposition. When we
refer to a one pointed end action we are assuming that the source is
given by an isolated point.

Using BB decomposition we can describe the Picard group of our
varieties in terms of the source of the action. 

\begin{prop}\label{onepointend-basic}
Let us keep the same notation of Theorem \ref{thm_ABB-decomposition}. Suppose that a $\CC^*$ action on a smooth variety $X$ has
  one pointed end with source $y_0$. Then $X$ is rational, and $\Pic{X}$ is finitely
  generated with no torsion.  Moreover, the divisors
  $D_i^+=\overline{X^+(Y_i)}$ for $Y_i$ such that $\rk^-(Y_i)=1$, are
  irreducible and their classes make the basis of $\Pic X$.
  \end{prop}

 \begin{proof}


Applying Theorem \ref{thm_ABB-decomposition}, we get $H_{2}(X,\mathbb{Z})=\bigoplus_{i\in I} H_{2-2 \rk^{-}(Y_i)}(Y_i,\ZZ)$. Being $y_0$ an isolated point with $\rk^{-}{(y_0)}=0$, then the only fixed components $Y_i$ which contribute to the homology are those having $\rk^{-}(Y_i)=1$. Therefore $H_{2}(X,\mathbb{Z}) \iso \ZZ^{\rho}$ with $\rho \in \mathbb{Z}_{\geq 0}$. By Theorem \ref{thm_ABB-decomposition}, one has $X^+(y_0)\iso\CC^n$, hence $X$ is rational. In particular, being $X$ simply connected, one has $$\Pic X\iso H^2(X,\mathbb{Z})\iso \text{Hom}_{\mathbb{Z}}(H_2(X,\mathbb{Z}),\mathbb{Z})\iso \ZZ^{\rho}$$
and our claim follows.

\end{proof}
\subsection{Linearization} \label{linearization} Let $p\colon
\cL\rightarrow X$ be a line bundle over a normal projective variety
with an action of an algebraic torus $H=(\CC^*)^r$. We recall that a
linearization $\mu$ of $\cL$ is an $H$ equivariant action on $\cL$
which is linear on the fibers of $p$, that is for every $t\in H$ and
$x\in X$ the restriction $\mu\colon \cL_{x}\to \cL_{t\cdot x}$ is
linear.  In this case we say that $(\cL,\mu)$ is an $H$ linearized
line bundle on $X$. See \cite[$\S$1.3]{MFK}, \cite[$\S$2.2]{BRION} or
\cite[$\S$2]{KNOP} for details on linearizations. From now on, we
denote by $\mu_{\cL}$ or simply by $\mu$ a chosen linearization of the
line bundle $\cL$.

By \cite[Proposition 2.4]{KNOP} and the subsequent Remark in \cite{KNOP}, we know
that there exists a linearization of the action of an algebraic torus $H$ on
$\cL$. Using \cite[Lemma 3.2.4]{BRION} we deduce that given two line bundles
$\cL_1$ and $\cL_2$ with linearizations $\mu_{\cL_1}$ and $\mu_{\cL_2}$,  their
product $\cL_1\otimes\cL_2$ has a natural linearization
$\mu_{\cL_1\otimes\cL_2}=\mu_{\cL_1}+\mu_{\cL_2}$, where for $H$ linearized line
bundles we will use the additive notation. Also the dual of any $H$ linearized
line bundle on $X$ is $H$-linearized as well. Thus the isomorphism classes of
$H$-linearized line bundles form an abelian group relative to the tensor
product, which we denote by  $\Pic^{H}(X)$. We have a short exact sequence:
\begin{equation}\label{linearization-exact-sequence}
\xymatrix{0\ar[r]&\Hom(H,\CC^*)=M\ar[r]^{\ \ \ \ \gamma}& \Pic^{H}(X)\ar[r]^{\
\varphi}&\Pic(X)\ar[r]&0}
\end{equation}
where $\varphi$ forgets the linearization and $\gamma$ are linearizations of the
trivial bundle. In particular, any two linearizations of a line bundle differ by
a character. Moreover, $TX$ and $\Omega X$ have natural linearizations, hence
$K_X$ too.

\begin{rem}\label{action_on_sections} Given a line bundle $\cL\to X$ with a linearization $\mu:H\times \cL\rightarrow \cL$ we get the
  action on $\HH^0(X,\cL)$ such that
  $$H\times\HH^0(X,\cL)\ni(t,\sigma)\longrightarrow
  \left(x\mapsto (t\cdot\sigma)(x):=\mu(t,\sigma(t^{-1}\cdot
    x))\right)\in\HH^0(X,\cL)$$ for every $\sigma \in \HH^0(X,\cL)$,
  $t\in H$, and $x\in X$.  When $\cL$ is semiample we can consider the
  graded finitely generated $\CC$-algebra $\cR=\bigoplus_{m\geq
    0}\HH^0(X,m\cL)$. As we have already observed, each line bundle
  $m\cL$ has an induced linearization, then there is an induced $H$
  action on $\HH^0(X,m\cL)$, hence on $\cR$.
\end{rem}

\begin{prop}\label{descending_action}
  Let $(X,L)$ be as in Assumptions \ref{gen-assumpt}. Then the target
  of the adjunction morphism $\phi_{\tau}: X\rightarrow X'$ admits an
  action of $\CC^*$ (possibly non effective) such that $\phi_{\tau}$ is
  $\CC^*$ equivariant.
\end{prop}
\begin{proof} Taking the natural linearization for $K_{X}$ it follows
  that $K_{X}+\tau L$ admits a linearization, and by Remark
  \ref{action_on_sections} we deduce that the torus acts on the
  variety given by (\ref{proj}). In fact, taking
  sufficiently large multiple $m$ of the divisor $K_{X}+\tau L$, we may
  assume that $m(K_{X}+\tau L)$ is the pullback of a very ample
  divisor on $X'$, and the action on $X'$ is induced by equivariant
  embedding into $\PP(\HH^0(X,m(K_{X}+\tau L))$.
\end{proof}

The following construction was used in \cite[$\S$2.1]{B_W}.  Let $X$
be a normal projective variety with an action of an algebraic torus
$H$ of rank $r$ whose set of fixed point component is $\cY$. Let us
consider a linearization $\mu_{\cL}$ of a line bundle $p\colon \cL\to
X$, and $Y\in \cY$. Given $y\in Y$ we associate $\mu_{\cL}(y)\in M=
\Hom_{alg}(H,\CC^*)\cong \ZZ^r$ which is the weight of the action of
$H$ on $p^{-1}(y)$. If $y_{1}$, $y_{2}$ belong to the same connected
component $Y$, then $\mu_{\cL}(y_{1})=\mu_{\cL}(y_{2})$, and we will
denote this weight by $\mu_{\cL}(Y)$. In this way we get a homomorphism
of abelian groups $\Pic^{H}(X)\rightarrow {M}^\cY$, with ${M}^\cY$ denoting
the additive group of functions $\cY\rightarrow M$, which to
linearized line bundle $(\cL,\mu_\cL)$ associates the function
$$\cY\ni Y\mapsto \mu_\cL(Y)\in M$$

\begin{defi}\label{mu-map}
  The above constructed function, which by abuse we continue to denote
  by $\mu_\cL$ (or simply by $\mu$), will be called \textit{fixed
    point weight map} $$\mu_\cL:\cY\ra M = \Hom_{alg}(H,\CC^*)\iso\ZZ^r$$
\end{defi}


\begin{rem}\label{rem-changing-torus}
  Suppose that an algebraic torus $H$ acts on $X$ and it contains an
  algebraic subtorus $\iota: H'\rightarrow H$. Then the action of $H$
  induces via $\iota$ the action of $H'$. Given any line bundle $\cL$
  over $X$ with $H$ linearization $\mu_\cL$, we have a unique induced
  linearization $\mu_\cL'$ of the action of $H'$. Moreover, we have
  the inclusion of the fixed point locus $X^{H}\subset X^{H'}$, and hence
  the map of the fixed point components $\iota_\bullet: \cY\rightarrow
  \cY'$. Then for the associated fixed point weight maps one has
  $$\mu_\cL'\circ\iota_\bullet=\iota^*\circ\mu_\cL$$
  where $\iota^*:M\rightarrow M'$ is the homomorphism of lattices of
  characters of the respective tori.
\end{rem}

In the case of a $\CC^*$ action on $X$, we distinguish the sink $Y_\infty$ of the action and say that the
linearization is \textit{normalized} if $\mu_\cL(Y_{\infty})=0$. That is,
a normalized line bundle $(\cL,\mu_\cL)$ is in the kernel of the homomorphism
$$\Pic^{\CC^*}(X)\ni (\cL,\mu_\cL)\mapsto\mu_\cL(Y_\infty)\in\ZZ.$$ In other words, the
choice of a normalized linearization splits the exact sequence
(\ref{linearization-exact-sequence}).

Using the map $\mu_\cL$ for a $\CC^*$ action, we introduce the following new definition.

\begin{defi}\label{bandwidth}
  Let $X$ be a normal projective variety admitting a $\CC^*$ action. Suppose that $\cL$ is a nef line bundle over $X$ with the
  fixed point weight map $\mu_\cL\colon \cY\ra \ZZ$. We denote by $\mu_{\min}$ and $\mu_{\max}$ the minimal and
  maximal value of $\mu$. The \textit{bandwidth} $|\mu|$ of the triple
  $(X,\cL,\CC^*)$ is $|\mu|:=\mu_{\max}-\mu_{\min}$. For short, we also say that $X$ and $\cL$ have bandwidth $|\mu|$.
\end{defi}

\section{Adjunction, Mori theory for varieties with a $\CC^*$ action}%
\label{section-adjunction}
In this section we describe the main ideas regarding adjunction theory for
varieties with a $\CC^*$ action.
\subsection{AM vs FM}\label{subsect-AMvsFM}
We begin with an easy example which we discuss in detail. Then we will prove AM vs FM equality in Lemma \ref{AMvsFM}. We refer to \cite[$\S$2.3]{OSCRW} for some
consequences of this equality, and for its generalization to vector bundles.
\begin{ex}\label{example-P1}
  Let us consider the standard action of $\CC^*$ on $\PP^1$ which in
  homogeneous coordinates $[x_0,x_1]$ is defined as follows
  $$H\times\PP^1\ni (t,[x_0,x_1])\longrightarrow t\cdot [x_0,x_1]=[tx_0,x_1]\in\PP^1$$

  The two fixed points are $y_0=[0,1]$ and $y_{\infty}=[1,0]$ with
  local coordinates $\frac{x_0}{x_1}$ and $\frac{x_1}{x_0}$ on which
  $\CC^*$ acts with weights $+1$ and $-1$, respectively.  If we write
  $t=\frac{x_0}{x_1}$ and $t^{-1}=\frac{x_1}{x_0}$, then the action
  extends the action of $\CC^*$ on itself. Moreover, if $y\in
  \PP^1\setminus\{y_0,y_{\infty}\}$ then $\lim_{t\rightarrow 0} t\cdot
  y = y_0$ and $\lim_{t\rightarrow \infty} t\cdot y =
  y_{\infty}$. Thus $y_0$ is the {\em source} and $y_{\infty}$ is the
  {\em sink} of the action.

  We recall that the universal bundle $\cL=\cO(-1)$ is embedded into
  the trivial bundle $V\times\PP^1$, where $V$ is the vector space
  with coordinates $(x_0,x_1)$ and $V\setminus\{0\}\rightarrow \PP^1$
  is the projection. The vector space has the obvious $\CC^*$ action
  $t\cdot(x_0,x_1)=(tx_0,x_1)$, and the composition
  $$\cL\hookrightarrow V\times\PP^1\longrightarrow \PP^1$$
  is $\CC^*$ equivariant. The fiber of $\cL\rightarrow\PP^1$ over
  $y_{\infty}=[1,0]$ is a line with coordinate $x_0$, and over
  $y_{0}=[0,1]$ is a line with coordinate $x_1$. This yields a
  linearization $\mu_\cL$ of $\cL$ such that $\mu_\cL(y_{\infty})=1$
  and $\mu_\cL(y_0)=0$. Then, if we replace $\cL=\cO(-1)$ with
  $\cL^{\vee}=\cO(1)$ we get $\mu_{\cL^{\vee}}(y_{\infty})=-1$ and
  $\mu_{\cL^{\vee}}(y_0)=0$.
\end{ex}

\begin{lemma}\label{AMvsFM}
Let $\CC^*\times\PP^1\rightarrow\PP^1$ be an effective
action with fixed points $y_0$ and $y_\infty$, which are
respectively the source and the sink of the action. Consider a line bundle $\cL$
over $\PP^1$ with linearization $\mu_\cL$. Then
$$\mu_\cL(y_0)-\mu_\cL(y_\infty)=\delta(y_0)\cdot \deg\cL$$ where $\delta(y_0)$
denotes the weight of the $\CC^*$ action on the tangent space
$T_{y_0}\PP^1$.
\end{lemma}

\begin{proof}
  If $\cL:= \cO(-1)$ and the action is standard, then the statement
  follows by Example \ref{example-P1}. As observed in Subsection
  \ref{linearization}, a linearization of a line bundle implies a
  linearization of its multiples and of its dual. Similarly, a
  multiple of the standard action multiplies the weights, both
  $\delta$ and $\mu$.  Hence the claim follows.
\end{proof}

Now let us apply the above observation to any manifold $X$ with a $\CC^*$
action.  Given a nontrivial orbit $\CC^*\cdot
x\hookrightarrow X$ and its closure $C\subset X$ we can take either a
normalization $f: \PP^1\rightarrow C\subset X$ or a
\textit{parametrization} $f_{\CC^*}:\PP^1\rightarrow C\subset X$. The latter
is defined by the formula $f_{\CC^*}(t)=t\cdot x$ for $t\in
\CC^*$, so that the action of $\CC^*$ on
$\PP^1$ is standard.

The morphism $f_{\CC^*}$ factors through the normalization $f$:
$$\xymatrix{\PP^1\ar[dr]^{f_{\CC^*}}\ar[d]_{\pi_\delta}&&\\ \PP^1\ar[r]^{f\ \ \
}&C\subset X}$$ That is $f_{\CC^*}=\pi_\delta\circ f$ where $\pi_\delta$ is
$\CC^*$ equivariant cover $\PP^1\rightarrow\PP^1$ of degree $\delta$ associated to
the weight of the $\CC^*$ action on the tangent space $T_{y_0}\PP^1$.
Equivalently, $\delta$ is the order of stabilizer of $x$ in $\CC^*$ acting on $X$.

Finally, if the action of $\CC^*$ on $X$ is equalized then, by the local
description of the action around the fixed point components (see
Theorem \ref{thm_ABB-decomposition}), we conclude that $C$ is smooth
and $f=f_{\CC^*}$.

Having the above in mind we obtain the following:

\begin{cor}\label{AMvsFMonX}
  Let $X$ be a smooth variety with an effective $\CC^{*}$ action, and let $f\colon
  \PP^1\to X$ be a non-constant $\CC^*$ equivariant map. Let $y_\infty$
  and $y_0$ be respectively the sink and source of the action on $\PP^1$.  Take $\cL$ a line bundle on $X$ with linearization $\mu_{\cL}$. Then the following hold:
\begin{itemize}[leftmargin=*]
\item[(a)] $\deg{f^{*} \cL}$ has the same sign (or it is zero) as the
  difference $$\mu_{\cL}(f(y_0))-\mu_{\cL}(f(y_{\infty})).$$
\item[(b)] If $\cL$ is nef and the action of $H$ is faithful, then the bandwidth
of the triple $(X,\cL,\CC^*)$ is equal to the degree of $\cL$ on the closure of a
general orbit of $\CC^*$.
\item[(c)] If the action of $\CC^*$ is equalized and $f$ is the
  normalization of the closure of a nontrivial orbit $C\subset X$,
  then $\deg{f^{*} \cL}=\mu_{\cL}(f(y_0))-\mu_{\cL}(f(y_{\infty}))$.
\end{itemize}
\end{cor}
\begin{ex} \label{ex_no_equalized}
In what follows, we discuss an easy example of a non-equalized  action which explains the assumptions in the
  preceding corollary. Let us consider an action of $\CC^*$ on $\PP^2$ with weights $(0,1,2)$,
that is
$$\CC^*\times\PP^2\ni(t,[z_0,z_1,z_2])\rightarrow [z_0,tz_1,t^2z_2]\in\PP^2$$
with three fixed points $y_0=[1,0,0],\ y_1=[0,1,0],\ y_2=[0,0,1]$, where $y_0$ is the source of this action, and $y_2$ is the sink. If
$L=\cO(1)$ then $\mu_L(y_i)=-i$ for $i=0,1,2$, so that $(\PP^2,L,\CC^*)$ has bandwidth two. Lines $z_0=0$ through $y_1$ and $y_2$, and $z_2=0$ through $y_0$ and $y_1$ are closures of orbits with the standard action of
$\CC^*$. Take the line $z_1=0$ through $y_0$ and $y_2$. By Lemma \ref{AMvsFM}, we deduce that this line is the closure of the
orbit with the action of $\CC^*$ of weight 2 (and thus the isotropy of rank
two), so that Corollary \ref{AMvsFMonX} (c) fails.  A general orbit is a conic $z_0z_2=a\cdot z_1^2$ with $a\ne 0$.
\end{ex}

\subsection{Graph of the action, cone theorem}
Let us start this subsection with a simple version of the localization
theorem, see e.g.~\cite{EDIDIN,QUILLEN} for more details.
We are interested in the description of $\Pic X$ in terms of normalized
linearizations.

\begin{prop}\label{Pic_vs_linearizations}
  Let $X$ be a projective manifold with an action of the torus
  $\CC^*$. Take the decomposition of the fixed locus into
  irreducible components $X^{\CC^*}=\bigsqcup_{i\in I}Y_i$ with $Y_\infty$
  denoting the sink component. Suppose that any effective curve on $X$
  is numerically equivalent to a sum of closures of orbits of
  $\CC^*$. Consider a function $\Upsilon: \Pic X
  \rightarrow\bigoplus_{Y_i\ne Y_\infty}\ZZ\cdot Y_i$ such
  that $$\Upsilon(\cL)=\sum_{Y_i\ne Y_\infty}\mu^\infty_\cL(Y_i)\cdot
  Y_i$$ where $\mu^\infty_\cL$ is the normalized linearization of
  $\cL$, i.e. $\mu^\infty_\cL(Y_\infty)=0$. Then $\Upsilon$ is a
  homomorphism of groups with the kernel equal to numerically trivial
  line bundles.
\end{prop}
\begin{proof} In the discussion following Remark \ref{rem-changing-torus},  we
  noted that normalized linearization splits the sequence
  (\ref{linearization-exact-sequence}). Therefore, $\Upsilon$ is the composition
  $$\Pic X\rightarrow \Pic^{\CC^*}X\rightarrow
  \ZZ^\cY\rightarrow\bigoplus_{Y_i\ne Y_\infty}\ZZ\cdot Y_i$$ where
  the arrow in the middle is the fixed point weight map, and the right
  arrow is the projection. Thus $\Upsilon$ is a homomorphism of
  groups. By Lemma \ref{AMvsFM}, if $\mu^0_{\cL}$ is zero then the
  degree of $\cL$ on the closure of every orbit is zero, hence the
  claim. \end{proof}

\begin{defi}\label{graph}
  Let $X$ be a smooth projective variety with an effective $\CC^*$ action. We define a \textit{directed graph}
  $\cG=\cG(X,\CC^*):=(\cY,\cE)$ with the set of vertices being the set of
  the fixed point components $\cY=\{Y_i\}$ of the $\CC^*$ action, and
  the set of directed edges $\cE$ defined as follows:
  $\epsilon(Y_{i_1},Y_{i_2})=\overrightarrow{Y_{i_1}Y_{i_2}}\in\cE$ is
  a directed edge joining components $Y_{i_1}, Y_{i_2}\in\cY$ if and
  only if there exists a nontrivial orbit $\CC^*\cdot x$ such
  that $\lim_{t\rightarrow 0} t\cdot x\in Y_{i_1}$ and
  $\lim_{t\rightarrow \infty} t\cdot x\in Y_{i_2}$. Note that this
  edge is directed from $Y_{i_1}$ to $Y_{i_2}$ and by
  (\ref{partial_order}) we have $Y_{i_1}\prec Y_{i_2}$. In this case,
  we say that the fixed point components $Y_{i_1}$ and $Y_{i_2}$ are
  joined by an orbit of the $\CC^*$ action; $Y_{i_1}$ precedes
  $Y_{i_2}$, and $Y_{i_2}$ succeeds $Y_{i_1}$ in the graph
  $\cG$. \end{defi}

\begin{ex} This is an extension of Example \ref{example-P1}. Let us
  consider an action of $\CC^*$ on a vector space $W$ of
  dimension $d=d_1+\cdots+d_s$, with $d_j>0$, given by weights
  $a_1>\cdots >a_s$, and eigenspaces of dimensions $d_1,\dots,
  d_s$. Namely, if $t\in \CC^*$ then in some coordinates on $W$ we
  have $$\begin{array}{l} t\cdot
    (z_1,\cdots, z_{d_1},z_{d_1+1},\cdots, z_{d_1+d_2},\cdots)=\\
    \phantom{t}(t^{a_1}z_1,\cdots, t^{a_1}z_{d_1},t^{a_2}z_{d_1+1},
    \cdots, t^{a_2}z_{d_1+d_2},\cdots)
\end{array}$$

 The $\CC^*$ action descends to $\PP^{d-1}$, the quotient of $W$ via homotheties.
 The fixed locus of this action has $s$ components $Y_1\iso\PP^{d_1-1},\dots$,
 $Y_s\iso\PP^{d_s-1}$ associated to eigenspaces of weights $a_1,\dots, a_s$
 respectively.  The action of $\CC^*$ on the fiber of
 $W\setminus\{0\}\rightarrow\PP^{d-1}$ over $Y_i$ is of weight $a_i$.  Thus, the
 induced linearization $\mu_L$ of the ample line bundle $L=\cO(1)$ maps $Y_i$ to
 $-a_i$. The graph $\cG$ is a complete graph with vertices in
 $\cY=\{Y_1,\dots,Y_s\}$ directed, so that we have
 \begin{equation}\label{order-in-projective-space}
 \epsilon(Y_{i_1},Y_{i_2})=\overrightarrow{Y_{i_1}Y_{i_2}}\in\cE \ \Leftrightarrow \
 a_{i_1}<a_{i_2}\ \Leftrightarrow \ \mu(Y_{i_1})>\mu(Y_{i_2}).
\end{equation}
 \end{ex}

We note that any polarized pair $(X,L)$ with a $\CC^*$ action
can be embedded equivariantly into some projective space $\PP^{N}$, so
that $mL$ is the restriction of $\cO(1)$, for some $m\gg
0$. Accordingly, the graph $\cG$ of fixed points and orbits for $X$ is
mapped to the graph of $\PP^N$.  Thus, in particular, the graph $\cG$
has no directed cycles nor loops.

An edge $\epsilon(Y_{i_1},Y_{i_2}) = \overrightarrow{Y_{i_1}Y_{i_2}}\in\cE$ is
called \textit{minimal} if there is no sequence of length $>1$ of directed edges
joining $Y_{i_1}$ to $Y_{i_2}$. The set of minimal edges for the graph
$\cG=(\cY,\cE)$ is denoted by $\cE^0$.

In the situation of Proposition \ref{Pic_vs_linearizations}, we consider a
vector space $\RR^{|\cY|-1}=\bigoplus_{i\ne\infty}\RR\cdot Y_i$ with
the dual basis of functionals $Y_i^*$. We define functionals
$\widehat{\epsilon}(Y_i,Y_\infty)=Y_i^*$ and
$\widehat{\epsilon}(Y_{i_1},Y_{i_2})= Y^*_{i_1}-Y^*_{i_2}$ for
$i_2\ne\infty$. For a functional $\widehat{\epsilon}$, we denote by
$\widehat{\epsilon}_{\geq 0}$ the halfspace on which the
functional is non-negative.

The following is an effective version of the nef cone for varieties with a $\CC^*$ action.

\begin{thm}\label{nef_cone_description} In the situation of Proposition
  \ref{Pic_vs_linearizations}, we assume that the cone of 1-cycles
  $\cC(X)$ is generated by classes of closures of orbits of the $\CC^*$
  action. Let us consider the map $\Upsilon_\RR:
  \NN^1(X)\rightarrow \bigoplus_{Y_i\ne Y_\infty} \RR \cdot Y_i$ which
  comes from the morphism defined in Proposition \ref{Pic_vs_linearizations}. Then
$$\Upsilon_\RR(\overline{\cA}(X))= \Upsilon_\RR(\NN^1(X)) \cap
\left(\bigcap_{\epsilon(Y_{i_1},Y_{i_2})\in\cE^0}
\widehat{\epsilon}(Y_{i_1},Y_{i_2})_{\geq 0}\right).$$ \end{thm}
\begin{proof}
  In view of Corollary \ref{AMvsFMonX}, since $\cC(X)$ is generated by
  classes of closures of orbits of $\CC^*$, we need to prove that a line bundle
  $\cL\in\Pic X$ is nef if and only if the fixed point weight map
  $\mu_\cL^\infty:\cY\rightarrow\ZZ$ is non-increasing on the vertices
  of the directed graph $\cG$.  That is, the partial linear order
  given by the function $\mu^\infty_\cL$ is opposite to the order $\prec$
  coming from the directed graph $\cG$. Given
  $\epsilon(Y_{i_1},Y_{i_2})\in\cE$, then $\mu_\cL^\infty(Y_{i_1})\geq
  \mu_\cL^\infty(Y_{i_2})$ if and only if
  $\widehat{\epsilon}(Y_{i_1},Y_{i_2})(\mu_\cL^\infty)\geq 0$. It
  is enough to check this inequality for the minimal edges
  $\epsilon(Y_{i_1},Y_{i_2})\in\cE^0$ to conclude the proof.
\end{proof}

\subsection{When orbits generate the cone of 1-cycles}
In Proposition \ref{Pic_vs_linearizations} and Theorem
\ref{nef_cone_description} we assume that the classes of closures of
orbits generate $\NN_1(X)$ and $\cC(X)$, respectively. On the other
hand, from Lemma \ref{curve-cone-generation} we know that $\cC(X)$ is
generated by the classes of closures of orbits and classes of curves
contained in the fixed locus of the action. Hence the assumptions of Proposition
\ref{Pic_vs_linearizations} and Theorem \ref{nef_cone_description} are
satisfied when the fixed locus consists of a finite number of points.  In this
subsection we extend this observation for a broader class of varieties
which turns out to be very useful in our applications.

\begin{lemma}\label{negligible-fix-pt-cmpt}
  Assume that $\CC^*$ acts effectively on a projective manifold $X$. Suppose that
  $Y$ is a connected component of the fixed locus of the action. Then
  the following conditions are equivalent:
\begin{enumerate}[leftmargin=*]
\item The component $Y$ is succeeded in the directed graph $\cG$ by one
  component consisting of a single point $y$.
\item The closure of the Bia{\l}ynicki-Birula cell $X^+(Y)$ adds a single point:
  $$\overline{X^+(Y)}\setminus X^+(Y)=\{y\}.$$
\item The positive weight subbundle $T^+$ of $TX_{|Y}$ is an ample
  line bundle and there exists an $H$ equivariant morphism:
  $$\PP_Y(T^+\!\oplus\cO)\longrightarrow X$$
  where the action of $H$ on the $\PP^1$-bundle has two fixed point
  components associated to two sections, $Y^0$ and $Y^\infty$. The
  section $Y^0$ has normal bundle $T^+$ and it is mapped
  isomorphically to $Y\subset X$; the section $Y^\infty$ has normal bundle
  $(T^+)^{\vee}$ and it is mapped to a point $y\in X$.
\end{enumerate}
\end{lemma}
\begin{proof}
  The implication $(2)\Rightarrow (1)$ is clear, because $X^+(Y)$
  contains all orbits whose source is in $Y$. Also the implication
  $(3)\Rightarrow (2)$ is obvious. Thus let us focus on the
  implication $(1)\Rightarrow (3)$.

  Take $L$ a very ample line bundle on $X$ and consider a $\CC^*$
  invariant divisor $D$ in $|L|$ which does not contain $y$. Every
  orbit $t\cdot x$ of $\CC^*$ such that $\lim_{t\to 0}t\cdot x\in Y$ has
  $\lim_{t\to\infty}t\cdot x=y$. Since the closure of every such orbit
  has intersection with $D$, it follows that $D\cap X^+(Y)=Y$ and, as a
  divisor on $X^+(Y)\iso T^+$ (c.f.~ Theorem
  \ref{thm_ABB-decomposition}), the restriction of $D$ is a multiple of
  the zero section in the bundle $T^+$, that is $D\cdot X^+(Y)=mY$ for
  some $m>0$. Thus $T^+$ is an ample line bundle over $Y$ on which $H$
  acts with a weight $\delta>0$. Moreover, since the argument does not
  depend on the choice of a very ample $L$, the restriction $\Pic
  X\rightarrow\Pic Y$ is contained in $\ZZ\cdot T^+$. On the other
  hand, because of Lemma \ref{AMvsFM}, the degree of any line bundle
  $L$ on the closure of every orbit joining $Y$ with $y$ is equal to
  $(\mu_L(Y)-\mu_L(y))/\delta$.

  The projective $\PP^1$-bundle $\pi:\PP_Y(T^+\!\oplus\cO)\ra Y$ has two sections
  associated to projections to two factors of the decomposable bundle. We denote
  the one with normal $T^+$ by $Y^0$ and the other one, whose normal is dual to
  $T^+$, by $Y^\infty$. Since $T^+$ is ample, we have a contraction morphism
  $$\PP_Y(T^+\!\oplus\cO)\longrightarrow\Proj(Sym(T^+\!\oplus\cO)):=\cS(Y,T^+)$$
  which contracts the section $Y^\infty$ to a point $y^\infty$ which is the
  vertex of the projective cone $\cS(Y,T^+)$. We define the action of $\CC^*$
  on $\PP_Y(T^+\!\oplus\cO)$ so that $Y^0$ and $Y^\infty$ is the source and
  sink, respectively, and along the fibers of the $\PP^1$-bundle the action has
  weight $\delta$. Therefore, we have a $\CC^*$ equivariant embedding
  $T^+\hookrightarrow \PP_Y(T^+\!\oplus\cO)$ with image equal to
  $\PP_Y(T^+\!\oplus\cO)\setminus Y^\infty = \cS(Y,T^+)\setminus\{y^\infty\}$.

  We claim that the $\CC^*$ equivariant isomorphism $T^+\iso X^+(Y)$ (see Theorem
  \ref{thm_ABB-decomposition}), extends to a regular $\CC^*$ equivariant morphism
  $$\PP_Y(T^+\!\oplus\cO)\longrightarrow\cS(Y,T^+)\longrightarrow X$$
  which has the properties as in $(3)$. Indeed, any $\CC^*$ invariant
  divisor in $|mY|$ on $T^+\iso X^+(Y)$ extends to
  $\PP_Y(T^+\!\oplus\cO)$ as the sum $aY^0+bY^\infty+\pi^*(M)$, where
  $a+b=m$ and $M\in |bT^+|$. Thus the desired extension exists and
  maps $Y^\infty$ to $y$.  We note that $\cS(Y,T^+)\rightarrow
  \overline{X^+(Y)}\hookrightarrow X$ is the normalization.
\end{proof}

We note that changing the direction of the action of $\CC^*$, and
therefore the direction of the graph $\cG$, we get a similar statement
as in the lemma above, with $0$ swapped with $\infty$, source with
the sink, and $T^+$ with $T^-$.
\begin{cor}\label{cor-negligible-fix-pt-cmpt2}
  Assume that $\CC^*$ acts effectively on a projective manifold $X$, and let $Y$ be a fixed point component which satisfies one of the equivalent conditions of Lemma \ref{negligible-fix-pt-cmpt}. Then the curves
  contained in $Y$ are numerically proportional to
  classes of closures of orbits joining $Y$ with $y$.
\end{cor}

\begin{proof} The corollary is a version of a known observation that curves in the base of a cone are numerically proportional to lines in the ruling of the cone.
We use Lemma \ref{negligible-fix-pt-cmpt} (3), and keep the same notation there introduced. For an irreducible curve $C$ contained in $\PP_Y(T^+)=Y^{\infty}$, one has $C\equiv \widetilde{C}+\alpha F$, where $\widetilde{C}$ is an irreducible curve contained in $\PP_Y(\cO)=Y^0$, $F$ is a fiber of the $\PP^1$-bundle $\PP_Y(T^+\!\oplus\cO)\to Y$, and $\alpha\in \QQ$. Mapping $\PP_Y(T^+\!\oplus\cO)$ to $X$, since $ \widetilde{C}$ is contracted to a point, we get the claim. 
\end{proof}

\subsection{Technical lemmata}
This last part of the present section contains technical lemmata which
will be used later in our applications.
\begin{lemma}\label{images_extremal_points} Let $\phi\colon X\to Z$ be a
surjective $\CC^*$ equivariant morphism of two normal projective varieties with an action of $\CC^*$. Suppose that $X$ is smooth and $Y_0$ (resp. $Y_\infty$) is the
source (resp. the sink) of $X$. Then for a general $z\in Z$ we have
$$\lim_{t\ra 0}t\cdot z\in\phi(Y_0)\ \ {\rm  and}\ \ \lim_{t\ra \infty}t\cdot
z\in\phi(Y_\infty).$$ \end{lemma}

\begin{proof} Since $\phi$ is equivariant, its restriction to $X^+(Y_0)$ or,
respectively, to $X^-(Y_\infty)$ dominates $Z$, and this implies the
claim. \end{proof}

\begin{cor} \label{cor_extremal_points} Let $\phi\colon X\to Z$ be a surjective
$\CC^*$ equivariant morphism of two normal projective varieties with an action of $\CC^*$. If $X$ is smooth and the action of $\CC^*$ on $X$ has one pointed end or
two pointed ends, then the action on $Z$ has at least one pointed end or two
pointed ends, respectively. \end{cor}

\begin{lemma}\label{intersection-dimension}
  Suppose that $\CC^{*}$ acts effectively on a projective manifold
  $X$. Let us consider two different components $Y_{1}$ and $Y_{2}\in
  \cY$. Assume that both $Y_1$ and $Y_2$ are succeded in $\cG$ by a
  single point component $\{y\}\in\cY$.  Then we have $$\dim
  Y_{1} + \dim Y_{2}\leq n-2$$
\end{lemma}

\begin{proof} First we observe that
  $\dim{\overline{X^+(Y_{1})}}+\dim{\overline{X^+(Y_{2})}}\leq n$,
  otherwise there would be an orbit passing through $y$ and belonging
  to $\overline{X^+(Y_{1})}\cap \overline{X^+(Y_{2})}$, against BB
  decomposition. On the other hand, because $Y_{i} \subsetneq
  \overline{X^+(Y_{i})}$ for $i=1,2$, we obtain that
  $$\dim{Y_{1}}+\dim{Y_{2}}+2\leq
  \dim{\overline{X^+(Y_{1})}}+\dim{\overline{X^+(Y_{2})}}\leq n$$ hence the
  claim.
\end{proof}


\section{Varieties with small bandwidth}\label{section-bandwidth}
In the present section we classify polarized varieties $(X,L)$ with an
effective $\CC^*$ action such that the bandwidth, or the degree of
the closure of a general orbit, is $\leq 3$.
\subsection{Bandwidth $\leq 2$}
The following has been proved in \cite[Proposition 3.12]{B_W}, we
reprove it using the notion of adjunction.

\begin{thm}\label{bwleq2classification}
  Let $(X,L)$ be a polarized pair satisfying Assumptions
  \ref{gen-assumpt}, with $\dim{X}=n$. Then
  \begin{itemize}
  \item if the sink of the action is an isolated point, and $|\mu|=1$
    then $(X,L)=(\PP^n,\cO(1))$;
  \item if $n\geq 2$ and the $\CC^*$ action has two pointed ends with $|\mu|=2$ then
    either $(X,L)=(\PP^n,\cO(1))$, or $(X,L)=(\mathcal{Q}^n,\cO(1))$. Moreover,
    the $\CC^*$ action is equalized only in the latter case.
\end{itemize}
\end{thm}
\begin{proof}
  Assume that $|\mu|=1$. Let $Y_{\infty}=\{y_\infty\}$ and $Y_0$ be
  respectively the sink and the source of the action. We can take
  $\mu(Y_{\infty})=0$, so that $\mu(Y_0)=1$. Applying \cite[Lemma
  3.11]{B_W} we deduce that $\mu_{K_X}(Y_{\infty})\geq n$ and
  $\mu_{K_X}(Y_0) < 0$. Therefore $$\mu_{K_X+nL}(Y_{\infty})\geq n >
  \mu_{K_X+nL}(Y_0)$$ and denoting by $C$ the closure of an orbit
  joining the source and the sink, by Corollary \ref{AMvsFMonX}
  $(a)$ we get the inequality $(K_X+nL) \cdot C<0$, so that $K_X+nL$ is not
  nef. Therefore, using Remark \ref{tau_integer} one has $\tau=n+1$,
  and by Theorem \ref{ionescu} we obtain that $(X,L)=(\PP^n,\cO(1))$.

  Similarly, in case $|\mu|=2$, $Y_\infty=\{y_\infty\}$,
  $Y_0=\{y_0\}$, we get $$\mu_{K_X+nL}(Y_{\infty})\geq n \geq
  \mu_{K_X+nL}(Y_0).$$ If $\mu_{K_X+nL}(Y_{\infty})>\mu_{K_X+nL}(Y_0)$,
  then as above we deduce that $K_X+nL$ is not nef, $\tau=n+1$ and
  applying Theorem \ref{ionescu} we get $(X,L)=(\PP^n,\cO(1))$. Assume
  that $\mu_{K_X+nL}(Y_{\infty})=n =\mu_{K_X+nL}(Y_0)$, and thus $K_X+nL$ is nef. Then the
  divisor $K_X+nL$ has intersection zero with a general orbit joining
  the source and the sink, so that the adjoint morphism
  $\phi_{n}$ contracts $X$ to a point.  Applying again Theorem \ref{ionescu}, which in this case coincides
  with a classical result by Kobayashi and Ochiai (see
  \cite{KOB}), we then deduce that either $(X,L)=(\PP^n,\cO(1))$, or $(X,L)=(\mathcal{Q}^n,\cO(1))$. In the first case, the $\CC^*$ action in coordinates will be $(t,[z_0,\dots,z_n])\rightarrow [z_0,tz_1,\dots ,tz_{n-1},t^2z_n]$. This action is not equalized because there exists an invariant line joining the sink and the source, and in view of Lemma \ref{AMvsFM} one has that the weight of the tangent bundle of this line at the sink is $2$ (the case $n=2$ has been discussed in Example \ref{ex_no_equalized}). Consider $(X,L)=(\mathcal{Q}^n,\cO(1))$. The action of a maximal torus $\widehat{H}$ on $(X,L)$ is described in \cite[Example 2.20]{B_W}. Now, the $\CC^*$ action is obtained from a downgrading of the action of $\widehat{H}$, that gives a projection of the corresponding lattice of characters $\pi\colon \widehat{M}\to \mathbb{Z}$. We observe that taking a component $Y\subset X^{\CC^*}$ corresponding to a weight $i\in \mathbb{Z}$, the weights of the $\CC^*$ action on $\cN_{Y/X}$ are obtained as projection of the weights of the $\widehat{H}$ action on the normal bundles at all the fixed components of $X^{\widehat{H}}$ which correspond to the weights that are sent to $i$ through $\pi$. Using the computations done in \cite[Example 2.20]{B_W} regarding the weights of the $\widehat{H}$ action on the cotangent bundles at the fixed components, we obtain the weights on the normal bundles at the same components, and using the previous observation we finally conclude that the $\CC^*$ action is equalized.
\end{proof}
The following two results concern the action of a torus on a quadric,
case of small bandwidth. Before this discussion, we recall by \cite[$\S$2.1]{B_W} that for any
polarized pair $(X,L)$ with an action of an algebraic torus $H$, the
\textit{polytope of fixed points} $\Delta(X,L,H,\mu_L)$ is the convex
hull of the image of the weight point map $\mu_L$ (see Definition \ref{mu-map}).
\begin{lemma}\label{quadric-bw1}
  Let $X$ be a smooth quadric of dimension $n\geq 3$.  Suppose that $\CC^*$ acts effectively on $X$ with fixed locus
  consisting of two components. Then both fixed point components are
  isomorphic to $\PP^m$, with $m=\lfloor\frac{n}{2}\rfloor$.
\end{lemma}
\begin{proof}
  Being $X$ a smooth quadric, $\CC^*$ is contained in some
  maximal torus $\widehat{H}$ of $SO_{n+2}$ with the lattice of
  characters $\widehat{M}=\bigoplus_{i=0}^m\ZZ e_i$. Thus, we are in
  situation described in \cite[Example 2.20]{B_W} and the action of
  $\CC^*$ is obtained from some downgrading $\CC^*\rightarrow\widehat{H}$ which
  comes with the homomorphism of lattices of characters
  $\widehat{M}\rightarrow \ZZ$.  We know that the polytope
  $\Delta=\Delta(\mathcal{Q}^n,\cO(1),\widehat{H})=\conv(\pm e_i, i=0,\dots, m)$
  has $2(m+1)$ vertices associated to fixed points of the action of
  $\widehat{H}$. Moreover, the polytope $\Delta$ is central symmetric
  and therefore its projection has the same property.  In view of Remark
  \ref{rem-changing-torus} all vertices of $\Delta$ are mapped via
  $\widehat{M}\rightarrow \ZZ$ to the set consisting of two points.
  Therefore the projection contracts two opposite facets of
  $\Delta(\mathcal{Q}^n,\cO(1),\widehat{H})$, each containing at least $m+1$
  vertices, thus both symplices associated to $\PP^{m}$. The last
  statement follows by \cite[Lemma 2.10]{B_W}.
\end{proof}

\begin{prop}\label{quadric-bw_leq2}
  Let $X$ be a smooth quadric of dimension $n\geq 3$ or a quadric
  cone, that is a cone over the smooth quadric of dimension $n-1$. By
  $L$ we denote the line bundle $\cO(1)$. Suppose that the torus
  $\CC^*$ acts effectively on $X$ with one pointed end, and the
  bandwidth of the action is $|\mu|\leq 2$. Then one of the following
  holds:
  \begin{enumerate}
  \item $X$ is a smooth quadric, $|\mu|=2$, the $\CC^*$  action has
    two pointed ends $y_0$ and $y_2$, and $X^{\CC^*}=\{y_0,y_2\}\sqcup
    \mathcal{Q}^{n-2}$.
  \item $X$ is a quadric cone and $X^{\CC^*}$ has two components: the vertex
    and a divisor $\iso\mathcal{Q}^{n-1}$.
  \item $X$ is a quadric cone and $X^{\CC^*}$ has three components: the
    vertex and two components $\iso\PP^m$, with
    $m=\lfloor\frac{n-1}{2}\rfloor$.
  \end{enumerate}
\end{prop}
\begin{proof}
  First, suppose that $X$ is the smooth quadric $\mathcal{Q}^n$, so that we are
  in situation described in \cite[Example 2.20]{B_W}. The torus $\CC^*$ is
  contained in some maximal torus $\widehat{H}$ of $SO_{n+2}$ with the
  lattice of characters $\widehat{M}$. Denote by $r$ the rank of $\widehat{H}$, and take $e_1,\dots, e_r$ a basis of $\widehat{M}$. By the downgrading, we see
  that the linearization of the action is associated to a projection
  $\pi\colon \widehat{M}\rightarrow \ZZ$. From \cite[Example 2.20]{B_W} we know
  that $\Delta(\mathcal{Q}^n,\cO(1),\widehat{H})=\conv(\pm e_i, i=1,\dots, r)$, and all the fixed points correspond to vertices of this polytope.  Therefore $\pi(\Delta(\mathcal{Q}^n,\cO(1),\widehat{H}))$ is a central
  symmetric polytope, and the action of $\CC^*$ has two pointed ends corresponding to $\pi(e_i)$, $\pi(-e_i)$ for some $i=1,\dots,r$. Moreover, the elements $\pm e_j$ with $j\neq i$ will be
  projected to the same point in $\ZZ$, so that $|\mu|=2$. We then obtain that the fixed point component associated to such a point is isomorphic to $\mathcal{Q}^{n-2}$, and this settles the smooth case.

  Now suppose that $X$ is a quadric cone. Let us choose a section of
  $L=\cO(1)$ which is $\CC^*$ equivariant and does not vanish at the
  vertex of the cone. Thus, the zero set $X'\subset X$ is a smooth
  quadric invariant with respect to the action of $\CC^*$. Then either
  $X'\in X^{\CC^*}$ and we get $(2)$, or the restriction of the action of
  $\CC^*$ to $X'$ has bandwidth $1$. In this latter case, applying Lemma
  \ref{quadric-bw1} to $X'$ we obtain $(3)$.
\end{proof}

\subsection{Bandwidth 3}
In this subsection we will study polarized pairs $(X,L)$ under the
following assumptions.

\begin{assumpt}\label{ass-bw3}
  Let $(X,L)$ be a polarized pair, where $X$ is a manifold of
  dimension $n\geq 2$ with a linearized $\CC^*$ action, such that
  it has two pointed ends and the bandwidth of $(X,L,\CC^*)$ is
  three. In addition, assume that the action is equalized.
\end{assumpt}

We start by discussing the easier case of surfaces, that was first studied in \cite[Example 3.16]{B_W}.

\begin{lemma} \label{BW3_surfaces} Assume that $(X,L)$ satisfies Assumptions
\ref{ass-bw3}, with $n=2$. Then either $(X,L)=(\PP^1\times \PP^1, \cO(1,2))$, or $(X,L)=(\PP_{\PP^1}(\cO(1)\oplus \cO(3)), \cO(1))$.
\end{lemma}

\begin{proof}
First, using either BB decomposition or the localization theorem (see proofs \cite[pag. 29-30]{B_W}) we get the description of the fixed point components, which are four isolated point, one for every weight $0,\dots,3$.  
Applying Proposition \ref{onepointend-basic} and its proof, we conclude that $X$ is rational, and $\Pic X\iso  \mathbb{Z}^2$. Therefore, $X\iso \PP_{\PP^1}(\cO\oplus \cO(e))$ for some $e\in \ZZ_{\geq 0}$. We will observe that either $e=0$ or $e=2$. 

Being $L$ ample on a Hirzebruch surface, we may write $L=aC_0+bF$, where $C_0$ is the minimal section, $F$ the fiber of the natural projection, $a>0$, and $b>ae$. The $\CC^*$ action on $X$ is obtained by a downgrading of a $(\CC^*)^{2}$ action on $X$.  The weights of this action on $L$ at the fixed points are the following: $(0,0)$; $(0,a)$; $(b-ae,a)$; $(b,0)$ (cf. \cite[Ex. 6.1.18]{Cox}).  Then, the weights of the $\CC^*$ action are obtained by a projection $\pi\colon \ZZ^{2}\to \ZZ$, such that $\pi(\bigtriangleup(X,L,(\CC^*)^{2}))=\bigtriangleup(X,L,\CC^*)$. A straightforward computation shows that the only cases in which we get a bandwidth three $\CC^*$ action having one fixed point corresponding to each lattice point, are those listed in the statement. 
\end{proof}

We now extend the study of bandwidth three varieties to high dimension. The following result will be the crucial point to prove results in
Subsection \ref{$SL_3$ action on contact manifolds}, and will be shown
in Section \ref{Class_bandwidth3}, by using adjunction theory when $n\geq 3$. Here, by inner fixed points
components we mean the components which are neither the sink nor the
source.

\begin{thm}\label{bw3classification}
  Suppose that $(X,L)$ satisfies Assumptions \ref{ass-bw3}.
  Then one of the following holds:
  \begin{enumerate}[leftmargin=*]
  \item $(X,L)=(\PP(\cV),\cO(1))$ is a scroll over $\PP^1$,
    with $L$ being relative $\cO(1)$ on the projectivisation of the vector
    bundle $\cV$ which is either $\cO(1)^{n-1}\oplus\cO(3)$ or
    $\cO(1)^{n-2}\oplus\cO(2)^2$. The inner fixed points components
    are two copies of $\PP^{n-2}$.
  \item $(X,L)=(\PP^1\times\mathcal{Q}^{n-1}, \cO(1,1))$ is a product quadric
    bundle over $\PP^1$. The inner fixed points components are two
    isolated points and two copies of $\mathcal{Q}^{n-3}$.
  \item $n\geq 6$ is divisible by 3 and $X$ is Fano, $\rho_X=1$,
    $-K_X=\frac{2}{3}nL$. The inner fixed points components are two
    smooth subvarieties of dimension $\frac{2}{3}n-2$.
  \end{enumerate}
\end{thm}

In the scroll case, we have the standard action of a rank $n$ algebraic torus on $X$; in the quadric bundle case one has the standard action of $\CC^*\times H_r$, with $\CC^*$ acting on $\PP^1$, and $H_r$ a maximal rank torus acting on $\mathcal{Q}^{n-1}$. In Examples \ref{ex-scroll}, \ref{ex-quad_bdl} we will see how the $\CC^*$ action are obtained from a downgrading of the standard action of the respective torus of bigger rank. In Example \ref{ex-sp6}, we present a variety satisfying Theorem \ref{bw3classification} (3). Notice that the classification of varieties satisfying part $(3)$ of the above theorem has been recently reached in \cite{OSCRW2}, using tools from birational and projective geometry. In total, there are four of these varieties, all of them are rational homogeneous; we refer to \cite[Theorem 6.8]{OSCRW2} for their complete list.

\subsection{Examples}\label{sect-bw3examples}
\begin{ex}\label{ex-scroll}
  Let us consider the standard action of $\CC^*$ on $\PP^1$ with
  source at $y_0$ and sink at $y_\infty$. For any line bundle $\cL$
  over $\PP^1$, we can choose its linearization so that
  $\mu_\cL(y_0)=a$ and $\mu_\cL(y_\infty)=a-\deg\cL$, where $a\in\ZZ$
  can be chosen arbitrarily, c.f. (\ref{linearization-exact-sequence})
  and Lemma \ref{AMvsFM}. Given a decomposable bundle $\cV$ over $\PP^1$, we
  can define its linearization by linearizing its components. If
  $\cV=\cO(1)^{n-1}\oplus\cO(3)$ then we linearize $\cO(1)$'s with
  $\mu(y_\infty)=1$ and $\mu(y_0)=2$, while the component $\cO(3)$
  is linearized so that $\mu(y_\infty)=0$ and $\mu(y_0)=3$. This
  determines the action of $\CC^*$ on $X=\PP(\cV)$ with the linearization of the relative $L=\cO(1)$.

  Alternatively, the pair $(X,L)$ can be described as a toric variety
  associated to a polytope $\Delta(L)$ in a lattice $M$ with
  generators $e_i$, $i=1,\dots, n$. We take the vertices of $\Delta$ as
  follows: $0, 3e_1$ and $e_1+e_i, 2e_1+e_i$ for $i>1$. The action of
  $\CC^*$ is defined by a downgrading $M\rightarrow \ZZ$ by the
  projection to the first coordinate.

  A similar construction works for $\cV=\cO(1)^{n-2}\oplus\cO(2)^2$.
  We linearize $\cO(1)$'s as before with $\mu(y_\infty)=1$ and
  $\mu(y_0)=2$, one copy of $\cO(2)$ with $\mu(y_\infty)=0$ and
  $\mu(y_0)=2$, and the other with $\mu(y_\infty)=1$ and $\mu(y_0)=3$.
  Or, alternatively we take $\Delta(L)$ in $M=\bigoplus_{i=1}^n\ZZ e_i$ with
  vertices follows: $0, 2e_1$ and $e_1+e_2, 3e_1+e_2$, and
  $e_1+e_i,2e_1+e_i$ for $i>2$. The action of $\CC^*$ is defined by
  downgrading $M\rightarrow \ZZ$ by the projection to the first
  coordinate.

  The fixed locus has four components: two extremal fixed points
  and two components isomorphic to $\PP^{n-2}$. The chosen
  linearization of the bundle $L$ associates to them the values $0,1,2,3$.

  We note that $K_X+nL=\pi^*\cO(n)$, with $\pi\colon \PP(\cV)\to
  \PP^1$ the natural projection, hence the nef value of the polarized
  variety $(X,L)$ is $n$ and $\pi$ is the adjunction map for $(X,L)$.

  In Figure \ref{fig-scroll} we present schematically the scroll
  situation: the thick black points and line segments are fixed point
  components, the thin line segments are orbits, and the shaded regions
  are fibers of the adjoint morphism over 0 and $\infty$.
\end{ex}
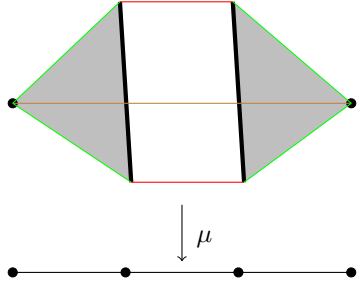
\begin{figure}[htbp]
\begin{tikzpicture}[scale=1.5]
\coordinate (0) at (0,0);
\coordinate (a1) at (0.95,0.9);
\coordinate (a2) at (1.95,0.9);
\coordinate (b1) at (1.05,-0.7);
\coordinate (b2) at (2.05,-0.7);
\coordinate (3) at (3,0);
\path [fill=lightgray] (0)--(a1)--(b1)--(0);
\path [fill=lightgray] (3)--(a2)--(b2)--(3);
\draw [fill] (0) circle [radius=0.04];
\draw [fill] (3) circle [radius=0.04];
\draw [ultra thick] (a1)-- (b1);
\draw [ultra thick] (a2)-- (b2);
\draw [brown] (0)--(3);
\draw [green] (0)--(a1);
\draw [green] (0)--(b1);
\draw [green] (3)--(a2);
\draw [green] (3)--(b2);
\draw [red] (a1)--(a2);
\draw [red] (b1)--(b2);
\draw [->] (1.5,-0.9)--(1.5,-1.4);
\node at (1.7,-1.2) {$\mu$};
\draw [fill] (0,-1.5) circle [radius=0.04];
\draw [fill] (1,-1.5) circle [radius=0.04];
\draw [fill] (2,-1.5) circle [radius=0.04];
\draw [fill] (3,-1.5) circle [radius=0.04];
\draw (0,-1.5)--(3,-1.5);
\end{tikzpicture}
\caption{Scroll case: fixed points, orbits, linearization}\label{fig-scroll}
\end{figure}

\begin{ex}\label{ex-quad_bdl}
  For $r\geq 2$, let $H_r$ be a torus $(\CC^*)^r$ with lattice of
  characters $M=\bigoplus_{i=1}^r\ZZ e_i$. The standard action of
  $H_r\subset SO_{2r}, SO_{2r+1}$ on the quadrics $\mathcal{Q}$ denoting
  $\mathcal{Q}^{2r-2}$ or $\mathcal{Q}^{2r-1}$ has a natural linearization on $\cO(1)$,
  so that $\Delta(\mathcal{Q},\cO(1),H_r)$ has vertices $\pm e_i$. Take $\PP^1$ with
  the standard action of $\CC^*$, and the linearization of $\cO(1)$
  with weights $(1,2)$. Consider $X=\PP^1\times\mathcal{Q}$ with the induced
  action of the product $\CC^*\times H_r$, and the lattice of
  characters $\ZZ e_0 \oplus M$. If $L=\cO(1,1)$ then the induced
  linearization yields $\Delta(\mathcal{Q}, L, \CC^*\times H_r)$ with vertices $e_0\pm e_i, 2e_0\pm
  e_i$ for $i>0$.

  Now, we take the action of $\CC^*$ on $X=\PP^1\times\mathcal{Q}$ which is
  obtained by the downgrading associated to the projection
  $\bigoplus_{i=0}^r \ZZ e_i\rightarrow\ZZ$ such that $e_0, e_1
  \mapsto 1$ and $e_i\mapsto 0$ for $i>1$.

  If $\dim X=3$, then $X=\PP^1\times\PP^1\times\PP^1$, $L=\cO(1,1,1)$,
  and the downgrading can be described in a symmetric way as a projection
  of a cube onto one of its diagonals. The action has 8 fixed
  points. We note that in this case $-K_X=2L$ and the associated
  adjoint morphism contracts $X$ to a point. In Figure
  \ref{fig-P1xP1xP1} we present schematically the fixed point set
  together with the orbits of the action, and the associated value of
  the linearization $\mu$ on the fixed point components.
\begin{figure}[htbp]
\begin{tikzpicture}[scale=1.8]
\coordinate (0) at (0,0);
\coordinate (1) at (0.95,0.6);
\coordinate (2) at (0.95,-0.6);
\coordinate (3) at (1.05,0.1);
\coordinate (4) at (1.95,-0.1);
\coordinate (5) at (2.05,0.7);
\coordinate (6) at (2.05,-0.5);
\coordinate (7) at (3,0);
\draw [fill] (0) circle [radius=0.04];
\draw [fill] (1) circle [radius=0.04];
\draw [fill] (2) circle [radius=0.04];
\draw [fill] (3) circle [radius=0.04];
\draw [fill] (4) circle [radius=0.04];
\draw [fill] (5) circle [radius=0.04];
\draw [fill] (6) circle [radius=0.04];
\draw [fill] (7) circle [radius=0.04];
\draw [brown, semithick] (0)--(7);
\draw [green] (0)--(1);
\draw [green] (0)--(2);
\draw [green] (0)--(3);
\draw [green] (7)--(4);
\draw [green] (7)--(5);
\draw [green] (7)--(6);
\draw [red] (1)--(5);
\draw [red] (1)--(4);
\draw [red] (2)--(6);
\draw [red] (2)--(4);
\draw [red] (3)--(5);
\draw [red] (3)--(6);
\draw [blue, ultra thin] (0)--(4);
\draw [blue, ultra thin] (0)--(5);
\draw [blue, ultra thin] (0)--(6);
\draw [blue, ultra thin] (7)--(1);
\draw [blue, ultra thin] (7)--(2);
\draw [blue, ultra thin] (7)--(3);
\draw [->] (1.5,-0.9)--(1.5,-1.4);
\node at (1.7,-1.2) {$\mu$};
\draw [fill] (0,-1.5) circle [radius=0.04];
\draw [fill] (1,-1.5) circle [radius=0.04];
\draw [fill] (2,-1.5) circle [radius=0.04];
\draw [fill] (3,-1.5) circle [radius=0.04];
\draw (0,-1.5)--(3,-1.5);
\end{tikzpicture}
\caption{$\PP^1\times\PP^1\times\PP^1$ with diagonal $\CC^*$ action}\label{fig-P1xP1xP1}
\end{figure}
\par\medskip

If $\dim X =n > 4$, then the induced action of $\CC^*$ has two pointed
ends and there are two fixed components associated to each weight 1
and 2: one is given by an isolated point, the other one is
isomorphic to $\mathcal{Q}^{n-3}$. In particular, for $n=4$ the two fixed
components associated to the weight 1 are an isolated point, and
another one isomorphic to $\PP^1$ with the restriction of $L$ being
$\cO(2)$, and the same holds for the fixed components associated to the
weight 2.

The nef value of the pair $(X,L)$ is $n-1$ with the adjoint map being
the projection $X\rightarrow \PP^1$.

In Figure \ref{fig-P1xQn} we present schematically the fixed point
locus together with the orbits of the action, and the associated value
of the linearization $\mu$ on the fixed point components. The two
shaded regions present the fibers of the adjoint morphism over the
fixed points of the $\CC^*$ action on $\PP^1$, that is $0$ and
$\infty$.
\begin{figure}[htbp]
\begin{tikzpicture}[scale=1.8]
\coordinate (0) at (0,0);
\coordinate (1) at (0.9,0.6);
\coordinate (2) at (0.95,-0.6);
\coordinate (3) at (1.1,0.1);
\coordinate (4) at (1.9,-0.1);
\coordinate (5) at (2.1,0.7);
\coordinate (6) at (2.05,-0.5);
\coordinate (7) at (3,0);
\path [fill=lightgray] (0)--(1)--(4)--(2)--(0);
\path [fill=lightgray] (7)--(6)--(3)--(5)--(7);
\draw [fill] (0) circle [radius=0.04];
\draw [fill] (3) circle [radius=0.04];
\draw [fill] (4) circle [radius=0.04];
\draw [fill] (7) circle [radius=0.04];
\draw [ultra thick] (1)--(2);
\draw [ultra thick] (5)-- (6);
\draw [brown, semithick] (0)--(7);
\draw [green] (0)--(1);
\draw [green] (0)--(2);
\draw [green] (0)--(3);
\draw [green] (7)--(4);
\draw [green] (7)--(5);
\draw [green] (7)--(6);
\draw [red] (1)--(5);
\draw [red] (1)--(4);
\draw [red] (2)--(6);
\draw [red] (2)--(4);
\draw [red] (3)--(5);
\draw [red] (3)--(6);
\draw [blue] (0)--(4);
\draw [blue, ultra thin] (0)--(5);
\draw [blue, ultra thin] (0)--(6);
\draw [blue, ultra thin] (7)--(1);
\draw [blue, ultra thin] (7)--(2);
\draw [blue] (7)--(3);
\draw [->] (1.5,-0.9)--(1.5,-1.4);
\node at (1.7,-1.2) {$\mu$};
\draw [fill] (0,-1.5) circle [radius=0.04];
\draw [fill] (1,-1.5) circle [radius=0.04];
\draw [fill] (2,-1.5) circle [radius=0.04];
\draw [fill] (3,-1.5) circle [radius=0.04];
\draw (0,-1.5)--(3,-1.5);
\end{tikzpicture}
\caption{$\PP^1\times\mathcal{Q}$ with $\CC^*$ action}\label{fig-P1xQn}
\end{figure}
\end{ex}
\par\medskip
\begin{ex}\label{ex-sp6} 
  Consider the $Sp_6$ homogeneous variety, namely the Lagrangian Grassmannian of isotropic planes in $\PP^5$. We verify that such a variety satisfies Theorem \ref{bw3classification} (3). It has dimension 6 and
  lives in $\PP^{13}$ (see \cite[$\S$17.1]{FULTON_HARRIS} for details about this variety). The
  representation of dimension 14 is associated to the highest weight
  $(1,1,1)$. The action of the big torus in $Sp_6$, which is of rank 3, 
  has 8 fixed points associated to the Weyl group orbit of the
  dominant weight. The weights associated to fixed points yield a cube
  in the weight space. We take the downgrading associated to the projection
  of the cube onto a long diagonal. The resulting $\CC^*$ action has
  fixed point locus which consists of two isolated points (the source
  and the sink) and two copies of $\PP^2$. A schematic picture is
  presented in Figure \ref{fig-Sp6}, with shaded triangles denoting
  the surface components of the fixed point set. The adjoint morphism contracts the variety to a point. We refer the interest reader to \cite[$\S$6]{OSCRW2} for the complete treatment and classification of all the varieties satisfying Theorem \ref{bw3classification} (3). 
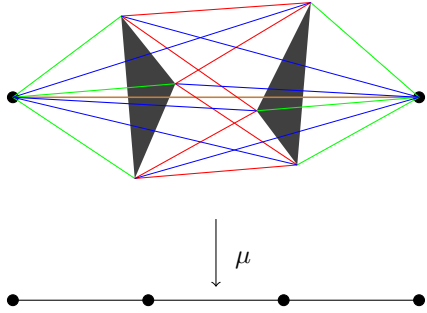
\begin{figure}[htbp]
\begin{tikzpicture}[scale=1.8]
\coordinate (0) at (0,0);
\coordinate (1) at (0.8,0.6);
\coordinate (2) at (0.9,-0.6);
\coordinate (3) at (1.2,0.1);
\coordinate (4) at (1.8,-0.1);
\coordinate (5) at (2.2,0.7);
\coordinate (6) at (2.1,-0.5);
\coordinate (7) at (3,0);
\path [fill=darkgray] (1)--(3)--(2)--(1);
\path [fill=darkgray] (6)--(4)--(5)--(6);
\draw [fill] (0) circle [radius=0.04];
\draw [fill] (7) circle [radius=0.04];
\draw [brown, semithick] (0)--(7);
\draw [green] (0)--(1);
\draw [green] (0)--(2);
\draw [green] (0)--(3);
\draw [green] (7)--(4);
\draw [green] (7)--(5);
\draw [green] (7)--(6);
\draw [red] (1)--(5);
\draw [red] (1)--(4);
\draw [red] (2)--(6);
\draw [red] (2)--(4);
\draw [red] (3)--(5);
\draw [red] (3)--(6);
\draw [blue] (0)--(4);
\draw [blue, ultra thin] (0)--(5);
\draw [blue, ultra thin] (0)--(6);
\draw [blue, ultra thin] (7)--(1);
\draw [blue, ultra thin] (7)--(2);
\draw [blue] (7)--(3);
\draw [->] (1.5,-0.9)--(1.5,-1.4);
\node at (1.7,-1.2) {$\mu$};
\draw [fill] (0,-1.5) circle [radius=0.04];
\draw [fill] (1,-1.5) circle [radius=0.04];
\draw [fill] (2,-1.5) circle [radius=0.04];
\draw [fill] (3,-1.5) circle [radius=0.04];
\draw (0,-1.5)--(3,-1.5);
\end{tikzpicture}
\caption{The $Sp_6$ homogeneous variety}\label{fig-Sp6}
\end{figure}

\end{ex}

\begin{rem} \label{isolated_points} In the case in which all fixed points
  are isolated points, we can apply BB decomposition and equivariant
  cohomology. Assume that $n\geq 3$. Under the Assumptions \ref{ass-bw3}, the
  equivariant Riemann-Roch gives the formula for
  $\chi_m(t)=\chi(X,L^m)$ (see \cite[Corollary A.3]{B_W}):
  $$\chi_m(t)=\frac{1}{(1-t)^n}+a\frac{t^m}{(1-t^{-1})(1-t)^{n-1}}
  +a\frac{t^{2m}}{(1-t)(1-t^{-1})^{n-1}}+\frac{t^{3m}}{(1-t^{-1})^{n}}$$
  where $a=\rank \Pic X$ is the number of fixed points associated to
  the weights 1 and 2.  In fact, because of the BB decomposition, $X$ has
  pure cohomology and $\chi(\cO_X)=1$; therefore
  $$\frac{1}{(1-t)^n}+a\frac{1}{(1-t^{-1})(1-t)^{n-1}}
  +a\frac{1}{(1-t)(1-t^{-1})^{n-1}}+\frac{1}{(1-t^{-1})^{n}}=1.$$ From
  this, multiplying by $(1-t)^n(1-t^{-1})^n$ we get the equality
  $$(1-t)^n+(1-t^{-1})^n+a(2-t-t^{-1})\cdot\left((1-t^{-1})^{n-2}+(1-t)^{n-2})\right)=
  (2-t-t^{-1})^n.$$ For $n=2$ we get $a=1$, while for $n=3$ we get
  $a=3$.  Let us assume $n\geq 4$ and write the highest terms of the
  left-hand side
  $$(-t)^n+(n+a)\cdot(-t)^{n-1}+\left(\binom{n}{2}+na\right)\cdot(-t)^{n-2}+\cdots$$
  while the highest terms of the right-hand side are
  $$(-t)^n+2n\cdot(-t)^{n-1}+\left(4\binom{n}{2}+n\right)\cdot(-t)^{n-2}+\cdots$$
  Comparing the second and the third term we see that for $n\geq 4$
  there is no solution.  For $n=a=3$ we have the case of
  $\PP^1\times\PP^1\times\PP^1$.
 \end{rem}


 \section{Classification of bandwidth 3
   varieties}\label{Class_bandwidth3}
 This section is devoted to prove Theorem
 \ref{bw3classification}. Let us keep Assumptions \ref{ass-bw3}, where
 we consider a normalized linearization $\mu$ so that we define
 $\cY_i:=\{Y\in \cY: \mu(Y)=i\}$ for $i=0,1,2,3$. All the fixed
 components in $\cY_1$ and $\cY_2$ are called \textit{inner
   components}. Using this notation, we will denote by $Y_0=\{y_0\}$
 and $Y_3=\{y_3\}$ respectively the sink and the source of the
 $\CC^{*}$ action. In view of Lemma \ref{BW3_surfaces}, from now on we suppose that $n\geq 3$.
 \begin{lemma}\label{bw3-components}
   In the situation of Assumptions \ref{ass-bw3} we have
   $\cY_1\ne\emptyset\ne\cY_2$. Moreover, in notation of Theorem
   \ref{thm_ABB-decomposition}, for every $Y_1\in\cY_1$, $Y_2\in\cY_2$
   we have $\rk^+(Y_1)=1=\rk^-(Y_2)$.
\end{lemma}
\begin{proof}
Firstly, arguing as in the proof of Proposition \ref{onepointend-basic}, we note that for $n>1$ we have
  $\cY_1\cup\cY_2\ne\emptyset$. So, contrary what lemma says, let us
  assume $\cY_2=\emptyset$ and take $Y\in\cY_1$. Then, because of
  Lemma \ref{negligible-fix-pt-cmpt}, one has $\rk^+(Y)=\rk^-(Y)=1$, hence
  $\dim Y=n-2$, and both $\overline{X^+(Y)}$ and $\overline{X^-(Y)}$
  are divisors. Thus, using Lemma \ref{intersection-dimension} we deduce that
  there are no other fixed point components in $\cY_1$, that is
  $\cY_1=\{Y\}$. Again, by Proposition \ref{onepointend-basic},
  divisors $\overline{X^+(Y)}$ and $\overline{X^-(Y)}$ are linearly
  equivalent and $\Pic X=\ZZ \cdot D$, where $D$ is their equivalence
  class. Moreover, if $C_1$ is the closure of an orbit with source at
  $Y$ and sink at $y_0$ and $C_2$ the closure of an orbit with source at
  $y_3$ and sink at $Y$, then $D\cdot C_1=D\cdot C_2=1$. However,
  because of Corollary \ref{AMvsFMonX}, $L\cdot C_1=1$ while $L\cdot
  C_2=2$, a contradiction. The last statement follows again by Lemma
  \ref{negligible-fix-pt-cmpt}.
\end{proof}


\subsection{Orbits}
The following is the graph of closures of possible orbits joining
fixed points components with $Y_1^i\in \cY_1$ and $Y_2^j\in \cY_2$. By
abuse, the orbits and their closures will be called by the same
name. We use the notation $A_*$ to denote that there could be
different orbits of type $A$ joining the component $Y_0$ (or $Y_3$)
with one of the components $Y_1^i\in \cY_1$ (respectively, with one of
the components $Y_2^j\in \cY_2$). In the same way, $B_*$ and $C_*$
denote the possible different orbits of type $B$ and $C$.

$$
\xygraph{
!{<0cm,0cm>;<3cm,0cm>:<0cm,1cm>::}
!{(0,0) }*+{\bullet_{Y_0}}="0"
!{(1,0) }*+{\bullet_{Y_1^i}}="1"
!{(2,0) }*+{\bullet_{Y_2^j}}="2"
!{(3,0)}*+{\bullet_{Y_3}}="3"
"0"-@/^/@[green]"1" _{A_*}
"2"-@/^/@[green]"3" _{A_*}
"0"-@/_1cm/@[blue]"2" _{B_*}
"1"-@/_1cm/@[blue]"3" _{B_*}
"1"-@/^/@[red]"2" _{C_*}
"0"-@/^1cm/@[brown]"3" ^{E}
} $$

\begin{rem}\label{orbits_A_E}
  For $n\geq 2$ orbits of type $A$ and $E$ always exist.  However, not
  all of the above types of orbits always exist:
\begin{enumerate}[leftmargin=*]
\item if $X=\PP^1\times\PP^1$, $L=\cO(1,2)$ then there are no orbits
  of type $C$,
\item if $X=\PP(\cO(1)^{n-1}\oplus\cO(3))$ then there are no orbits of
  type $B$.
\end{enumerate}
\end{rem}

In what follows, if not needed, we will not distinguish curves of
different types $A$ or $B$ and, if no confusion is probable, we
will write $d_*$ for the respective dimension of a component in
$\cY_*$.

\begin{lemma}\label{bandwidth3-intersection}
  Let $(X,L)$ be as in Assumptions \ref{ass-bw3} and let us keep the
  above notation for the possible orbits. The first two rows in the
  following table present the intersection of the closure of the orbit
  of the respective type (the column) with the divisor $L$ and $-K_X$.
  The third row presents the resulting estimate on $\tau$.
$$\begin{array}{c|c|c|c|c}
&\ A_*\ &\ B_*\ &\ C_*\ &\ \ E\ \ \\
\hline
L&1&2&1&3\\
\hline
-K_X&d_*+2&2n-d_*-2&2n-4-(d_1+d_2)&2n\\
\hline
\tau\geq &d_*+2&n-1-\frac{d_*}{2}&2n-4-(d_1+d_2)&\frac{2}{3}n\\
\hline
\end{array}$$
\par\noindent
For curves of type $A$ and $B$, by $d_*$ we denote the dimension of the
corresponding fixed source or sink component in $\cY_1$ or $\cY_2$. For curves of type $C$, by $d_2$ and $d_1$ we denote the
dimension of the sink/source component in $\cY_2$, $\cY_1$,
respectively.
\end{lemma}
\begin{proof}
  In view of Corollary \ref{AMvsFMonX}, the values of the first two rows
  are obtained by calculating the difference in $\mu_L$ and
  $\mu_{-K_X}$ at the source and the sink of each of the one
  dimensional orbits of the respective type. The values for $\mu_L$
  are known. Here, by $\mu_{-K_X}$ we denote the natural linearization
  of $-K_X$. Since the action is equalized, by \cite[Lemma 3.11]{B_W}
  we have $$\mu_{-K_X}(Y) = \rk^{+}(Y)-\rk^-(Y)$$ where we use the
  notation $\rk^\pm(Y)=\mathcal{N}^\pm(Y)$ as in Theorem
  \ref{thm_ABB-decomposition}. On the other hand, by Lemma
  \ref{negligible-fix-pt-cmpt} one has $\rk^+(Y_1)=\rk^-(Y_2)=1$ for
  $Y_1\in\cY_1, Y_2\in\cY_2$. This allows to compute the values of
  $\mu_{-K_X}$, and thus of the second row in the table. The third row
  is obtained by calculating the value of $K_X+\tau L$ from the two
  previous rows.
\end{proof}
\begin{prop}\label{tau-first_estimate}
  In the situation of Assumptions \ref{ass-bw3} we have the
  following:
  \begin{enumerate}[leftmargin=*]
  \item The space of 1-cycles $\NN_1(X)$ and the cone of curves $\cC(X)$
    are generated by classes of closures of orbits of the $\CC^*$ action.
  \item If $-K_X=\tau L$ then $\tau=\frac{2}{3}n$ and for any
    $Y\in\cY_1\cup\cY_2$ we have $\dim Y=\frac{2}{3}n-2$.
  \item $X$ is a Fano manifold unless there are components
    $Y_1\in\cY_1$, $Y_2\in\cY_2$, both of codimension 2, connected by
    an orbit of type $C$.
  \item If there exists a component of $X^H$ of codimension $2$, then
    $\tau \geq n$.
  \item If a component in $\cY_1\cup\cY_2$ does not meet a curve of
    type $C$, then it is of codimension 2.
  \end{enumerate}
\end{prop}
\begin{proof} Claim $(1)$ follows by Corollary
  \ref{cor-negligible-fix-pt-cmpt2}. To prove $(2)$ we note that
  the orbits of type $E$ are general, thus they always exist so, by Lemma
  \ref{bandwidth3-intersection}, using the column associated to $E$, we
  have $$-K_X\cdot E=2n=\tau L\cdot E=3 \tau$$ while using other
  entries in the table we get $2n-d_*-2=\frac{4}{3}n$. In order to
  show $(3)$ we use Lemma \ref{bandwidth3-intersection} again, and look at
  the intersection of orbits with $-K_X$, where we recall that $d_*\leq
  n-2$. Part $(4)$ follows by the existence of orbits of type $A$, and
  by the respective entries at the last row of the above table.
  Finally, assume that $Y\in\cY_1\cup\cY_2$ does not meet a curve of
  type $C$. Then, by Lemma \ref{negligible-fix-pt-cmpt} we get
  $\rank\cN^\pm(Y)=1$, hence $\dim{Y}=n-2$ and we obtain $(5)$.
\end{proof}

\subsection{$\tau\geq n$, scroll over a curve} \label{tau>=n} If
$\tau\geq n$, then because of Proposition \ref{tau-first_estimate}
$(2)$, we know that $(X,L)$ is neither $(\PP^n,\cO(1))$ nor
$(\mathcal{Q}^n,\cO(1))$.  Thus by Remark \ref{tau_integer} and Theorem
\ref{ionescu}, it follows that $(X,L)$ is a scroll over $\PP^1$. This is the first
claim in the following.

\begin{lemma}\label{scroll_onP1}
  If the nef value of the pair $(X,L)$ is $\geq n$ then this pair is a scroll
  over $\PP^1$ as described in case (1) of Theorem
  \ref{bw3classification}.
\end{lemma}
\begin{proof}
  By the above discussion, we deduce that $(X,L)$ is a scroll, so we
  know that there exist curves contracted by the adjoint morphism, passing through the end points. By looking at the table of Lemma
  \ref{bandwidth3-intersection}, we see that the intersection of
  $K_X+nL$ with curves of type $B_*$ and $E$ is positive, hence curves
  of type $A_*$ are contracted. Using the same table to compute the intersection number with such curves, we get $$(K_X+nL)\cdot A_*=-d_*-2+n L\cdot A_*=0.$$ Since by Corollary \ref{AMvsFMonX} (c) we know that $L\cdot A_*=1$, we deduce that $d_*=n-2$, namely there exist $Y_1\in\cY_1$ and
  $Y_2\in\cY_2$ which have dimension $n-2$. Hence $\rank
  \cN^{\pm}(Y_i)=1$, and being $\rho_X=2$, arguing as in the proof of Proposition \ref{onepointend-basic} we conclude that $X^{\CC^*}=\{y_0\}\sqcup
  Y_1\sqcup Y_2\sqcup \{y_3\}$.  Moreover, $\overline{X^+(Y_1)}$ and
  $\overline{X^-(Y_2)}$ are divisors and fibers of the adjoint
  morphism $\phi:=\phi_{n}\colon X\to \PP^1$, then they are isomorphic
  to $\PP^{n-1}$. Since by Lemma \ref{negligible-fix-pt-cmpt} these
  divisors are respectively cones over $Y_{1}$ and $Y_2$, and $X$ is a
  scroll, it follows that $Y_1\cong Y_2 \cong \PP^{n-2}$.
  
 We observe that there exist orbits of type $C$ joining $Y_1$ and
  $Y_2$, otherwise we reach a contradiction using Lemma
  \ref{intersection-dimension}. By Proposition \ref{onepointend-basic}, we deduce that $X$ is a rational 
 scroll, hence it has another contraction. Therefore, by Theorem
  \ref{nef_cone_description}, the curves of type $C$ generate the
  other ray of the cone $\cC(X)$ whose intersection with $-K_X$ is
  zero, as we see from the table in Lemma \ref{bandwidth3-intersection}.  Let
  $\cV=\phi_*L$, thus $X=\PP(\cV)$ and if we write
  $\cV=\cO(a_1)\oplus\cdots\oplus\cO(a_n)$ with $0< a_1\leq\cdots\leq
  a_n$, then the other contraction of $X$ contracts sections of
  $\phi$ associated to the smallest summand in this
  decomposition. Hence, $1=L\cdot C=a_1$ and because
  $K_X+nL=\phi^*\cO(\deg\cV-2)$ one has $0=K_X\cdot C=\sum a_i -n-2$
  from which we get both possible splitting types of $\cV$ as in
  Theorem \ref{bw3classification} (1).
\end{proof}

\begin{lemma}\label{components}
  Suppose that $|\cY_i|=1$ for either $i=1$ or $i=2$. Then either
  $\rho_X=1$ and $X$ is Fano of index $\frac{2}{3}n$ , or $(X,L)$ is a
  scroll over $\PP^1$ as in Lemma \ref{scroll_onP1}.
\end{lemma}
\begin{proof}
  Suppose that $|\cY_1|=1$. If $\rho_X> 1$, then by Proposition
  \ref{onepointend-basic} there exists a component $Y_2^j\in\cY_2$
  with $\rank\cN^+(Y_2^j)=1$, and since by Assumptions \ref{ass-bw3}
  one has $\rank\cN^{-}(Y_2^j)=1$, then this component is of
  codimension 2. Hence, by Proposition \ref{tau-first_estimate} $(4)$
  one has $\tau\geq n$, and by Lemma \ref{scroll_onP1} we obtain that
  $(X,L)$ is a scroll over $\PP^1$ as in Lemma \ref{scroll_onP1}. On
  the other hand, if $(X,L)$ is not such a scroll, then by what we
  have already proved it follows that $\tau<n$, and there is no
  component of $X^{\CC^*}$ of codimension $2$. Then $X$ is Fano because of
  Proposition \ref{tau-first_estimate} $(3)$, and by the claim $(2)$
  of the same Proposition its index is $\frac{2}{3}n$.
\end{proof}

\subsection{$\tau\leq n-1$, quadric bundle over a curve}
In this subsection we keep Assumptions
  \ref{ass-bw3} with $\tau\leq n-1$.
\begin{lemma}\label{b2_vs_components}
  If $\tau\leq n-1$ then $|\cY_1|=|\cY_2|=\rho_X$, every inner
  component is connected to another inner component by a curve of
  type $C$, the manifold $X$ is Fano, and the cone $\cC(X)$ is
  generated by classes of curves of type $A$ and $C$.
\end{lemma}
\begin{proof} Firstly, we note that if an inner component $Y$ is not
  connected to some other inner component by a curve of type $C$, then we are in the situation of Lemma \ref{negligible-fix-pt-cmpt} for both
  $y_3$ preceding $Y$ and $y_0$ succeeding $Y$; therefore $Y$ is of
  codimension 2 and Proposition \ref{tau-first_estimate}
  (4) gives a contradiction.  Now we prove that $\rho_X=|\cY_1|$; the equality
  $\rho_X=|\cY_2|$ follows by the same arguments. By Proposition
  \ref{onepointend-basic}, we know that $\rho_X\geq |\cY_1|$; if the
  inequality is strict then we argue as in the proof of Lemma
  \ref{components} to get a component $Y_2^j\in\cY_2$ with
  $\rank\cN^+(Y_2^j)=1$, which again leads to $\dim Y_2^j=n-2$, and by
  Proposition \ref{tau-first_estimate} (4) we reach a
  contradiction. The rest of the lemma follows by Proposition
  \ref{tau-first_estimate} (3) and Theorem \ref{nef_cone_description}.
\end{proof}

\begin{lemma}\label{tau>n-3}
Suppose $\rho_X>1$. Then $\tau\geq n-2$.
\end{lemma}
\begin{proof}
  By Lemma \ref{b2_vs_components}, we may assume that there are at
  least two components $Y_i^1\ne Y_i^2 \in \cY_i$ for each $i=1,2$
  and, moreover, we can choose these components so that $Y_2^i$ is
  connected to $Y_1^i$ via a curve of type $C$, for $i=1,2$. The
  latter follows by a standard argument on finding partial matching in
  a bipartite graph with vertices $\cY_1\sqcup\cY_2$. Using Lemma
  \ref{intersection-dimension} we get $d_i^1+d_i^2\leq n-2$, where
  $d_i^j=\dim Y_i^j$ for $i,j=1,2$. We confront this inequality with
  the estimate on $\tau$ for curves of type $C$ from Lemma
  \ref{bandwidth3-intersection} to get
  $$2\tau\geq 4n-8-(d_1^1+d_1^2+d_2^1+d_2^2)\geq 2n-4$$ hence the claim.
\end{proof}
\begin{rem}\label{tau=n-2=>equalities}
  From the proof of Lemma \ref{tau>n-3} we conclude that in case
  $\rho_X>1$ and $\tau=n-2$ all inequalities in the proof become
  equalities. That is, using the above notation
  $$d_1^1+d_1^2=d_2^1+d_2^2=n-2=d_1^1+d_2^1=d_1^2+d_2^2.$$
  In view of Lemma \ref{bandwidth3-intersection} the two
  right-hand side equalities imply that in this case the curves of
  type $C$ joining $Y_1^i$ with $Y_2^i$, for $i=1,2$, are contracted
  by $\phi_\tau$.
\end{rem}

\begin{rem} \label{ass_on_dimensions} Suppose that the adjoint morphism
  $\phi_{\tau}\colon X\to X^{\prime}$ is not the contraction to a
  point. Using Lemma \ref{bandwidth3-intersection} we get
  $\tau>\frac{2}{3}n$. Assuming $\tau\leq n-1$ we obtain $n\geq 4$, and
  if $\tau=n-2$ then $n\geq 7$.
\end{rem}

\begin{lemma} \label{tau>n-2}
  Assume that the adjoint morphism $\phi_{\tau}\colon X\to X^{\prime}$
  is not the contraction to a point. Then $\tau\geq n-1$.
\end{lemma}

\begin{proof} By Lemma \ref{tau>n-3} and Remark \ref{tau_integer}, we
  need to exclude the case $\tau=n-2$. We argue by contradiction and
  assume that $\phi_{n-2}\colon X\rightarrow X'$ is the adjoint
  morphism. We use Remarks \ref{tau=n-2=>equalities} and
  \ref{ass_on_dimensions}. By the former, we know that curves of type
  $C$ which join $Y^1_1\in \cY_1$ and $Y_2^1\in \cY_2$ are contracted
  by $\phi_{n-2}$. We may assume $d_1^1\geq d_1^2$, hence $d_1^1\geq
  3$ and $\dim \overline{X^+(Y^1_1)}\geq 4$. By fiber-locus inequality
  \cite[Theorem 1.1]{WIS} we deduce that fibers of $\phi_{n-2}$ have
  dimension $\geq n-3$, hence a fiber of $\phi_{n-2}$ has positive
  dimensional intersection with $\overline{X^+(Y_1^1)}$. Then
  $\overline{X^+(Y^1_1)}$ is contracted to a point by $\phi_{n-2}$,
  because of Corollary \ref{cor-negligible-fix-pt-cmpt2}. Thus
  $\phi_{n-2}$ contracts curves of type $A$ joining $y_0$ and $Y_1^1$,
  hence $d_1^1=n-4$ by Lemma \ref{bandwidth3-intersection}. Applying
  Remark \ref{tau=n-2=>equalities} we get $d_1^2=d_2^1=2$ and
  $d_1^1=d_2^2=n-4$. Using Lemma \ref{bandwidth3-intersection}, the
  curves in $\overline{X^+(Y^2_1)}$ and $\overline{X^-(Y_2^1)}$ are
  not contracted; therefore no fiber of $\phi_{n-2}$ of dimension
  $\geq n-2$ meets these subvarieties. Again, by \cite[Theorem
  1.1]{WIS} we conclude that $\phi_{n-2}:X\rightarrow X'$ is
  an equidimensional scroll over a smooth threefold; the smoothness
  follows from \cite[Lemma 2.12]{FUJITA-Sendai}.  The morphism
  $\overline{X^+(Y_1^2)}\rightarrow X'$ (and, in fact,
  $\overline{X^-(Y_2^1)}\rightarrow X'$) is finite and $\CC^*$ equivariant, 
  from which we infer that the $\CC^*$ action on $X'$ has two fixed
  point components, the image of $y_0$ and of $Y_1^2$, which is of
  dimension $2$. However, also $X\rightarrow X'$ is $\CC^*$ equivariant; thus Corollary \ref{cor_extremal_points} implies that the action
  of $\CC^*$ on $X'$ has two pointed ends, a contradiction.
\end{proof}

\begin{lemma} \label{dimX'=1}
  Suppose that $\tau=n-1$, and the adjoint morphism $\phi_{n-1}\colon
  X\to X^{\prime}$ is not the contraction to a point. Then $\dim X'=1$,
  and $\phi_{n-1}$ is a quadric bundle.
\end{lemma}

\begin{proof}
  Applying Theorem \ref{ionescu} (2), since by our assumption $X$ is not as described in point $(a)$, we are left to eliminate also cases $(c)$ and $(d)$ of that theorem. In the
  former case, because of Corollary \ref{cor_extremal_points}, there exists a
  fixed point $y'$ of the $\CC^*$ action on $X'$ which is not an end
  point. Hence, if $F\subset X$ is the fiber of $\phi_{\tau}$ over $y'$, 
  then $F$ is $\CC^*$ invariant with $\mu_{|\cY\cap F}$ assuming values
  only 1 or 2. In fact $F\iso\PP^{n-2}$, hence one of its fixed
  point components is of positive dimension and contained in, say, a
  component $Y_1\in\cY_1$.  We note however that by Corollary
  \ref{cor-negligible-fix-pt-cmpt2} the curves in $Y_1$ are
  numerically proportional to orbits joining the sink of the $\CC^*$
  action with $Y_1$. Hence the morphism $\phi_{n-1}$ contracts 
  $Y_1$ and also $\overline{X^+(Y_1)}$, which contains the sink. This
  contradicts the fact that $F$ does not contain any of the end points
  of the $\CC^*$ action on $X$.

  Suppose that $\phi_{n-1}:X\ra X'$ is birational; thus, by
  Theorem \ref{ionescu} it contracts at least one $F\iso\PP^{n-1}$ to a
  point. By the same arguments as above, we may assume that $F$
  contains one of the end points of the $\CC^*$ action on $X$, say
  $y_0\in F$. The action of $\CC^*$ on $F$ is of bandwidth $\leq 2$ and it
  is equalized. If the action of $\CC^*$ on $F$ is of bandwidth 1, then it
  has a fixed point component of dimension $n-2$, which by Proposition
  \ref{tau-first_estimate} implies $\tau\geq n$, not possible. If the
  $\CC^*$ action on $F$ is of bandwidth 2 then, by the same argument
  which was used in the first part of the proof, the other fixed end
  point component is an isolated point. Now, by Theorem
  \ref{bwleq2classification}, this is not possible if the action is
  equalized.
\end{proof}

\begin{lemma}\label{quadric-bdl-str}
  Suppose that $\tau=n-1$ and the adjoint morphism $\phi_{n-1}\colon
  X\ra \PP^1$ is a quadric bundle. Then $X=\PP^1\times\mathcal{Q}^{n-1}$,
  $L=\cO(1,1)$, and $\phi_{n-1}$ is the projection. Moreover, the sets
  of inner fixed points components $\cY_1$ and $\cY_2$ consist of an
  isolated point and a copy of $\mathcal{Q}^{n-3}$.
\end{lemma}

\begin{proof} By Remark \ref{ass_on_dimensions} we know that $n\geq
  4$. We describe $X^{\CC^*}$, by proving that for each $i=1,2$ the set of
  the fixed point components $\cY_i$ contains an isolated point and
  $\mathcal{Q}^{n-3}$. Let us take the quadrics corresponding to the fibers of
  $\phi_{n-1}$ over the end points of the induced action of $\CC^*$ on
  $\PP^1$; these fibers are either smooth quadrics or quadric
  cones. In view of Corollary \ref{b2_vs_components} and Proposition
  \ref{quadric-bw_leq2}, we deduce that both $\cY_1$ and $\cY_2$ have
  two components, say $\cY_i=\{Y_i^{1},Y_i^{2}\}$. Using Lemma
  \ref{intersection-dimension} we get $\dim{Y_i^1}+\dim{Y_i^2}\leq
  n-2$, and we may assume that $\dim{Y_i^1} \leq \dim{Y_i^2}$.  Notice
  that the fibers cannot be quadric cones. Suppose by contradiction
  that there is a fiber which is a quadric cone, then curves of type
  $B$ must be contracted by $\phi_{n-1}$, but by Lemma
  \ref{bandwidth3-intersection} we know that $K_X+(n-1)L$ has
  intersection positive with these curves. Thus, by Proposition
  \ref{quadric-bw_leq2} we deduce that $Y_i^1$ is an isolated point
  and $Y_i^2\cong \mathcal{Q}^{n-3}$ for each $i=1,2$, as claimed.

  Now we will prove that $X$ is a product. Using Proposition
  \ref{tau-first_estimate} we know that $X$ is Fano, then we can
  consider the other extremal contraction $\Psi\colon X\to Z$.
  By Theorem \ref{nef_cone_description}, we recall that the cone
  $\cC(X)$ is generated by classes of curves of type $A$ and $C$. More
  specifically, curves of type $A$ may join an end point, say
  $y_0$, with a component $Y_1^1$ which is a point, or with
  $Y_1^2\iso\mathcal{Q}^{n-3}$. In the former case, by Lemma
  \ref{bandwidth3-intersection} the intersection with $-K_X$ is $2$,
  in the latter $n-1$. Similarly, using Lemma \ref{bandwidth3-intersection}
  we verify that curves of type $C$ may have intersection with $-K_X$
  equal to $2, n-1$ and $2n-4$; the latter is not possible as $\tau=n-1$
  and $n\geq 4$.  Fibers of $\Psi$ are of dimension $\leq 1$ and have
  intersection $\geq 2$ with $-K_X$, hence applying \cite[Corollary
  1.3]{WIS} it follows that $\Psi$ is a $\PP^1$-bundle. Moreover
  $\Psi$ has two disjoint sections which correspond to the smooth
  quadrics that are fibers of $\phi_{n-1}$ over the two end points of
  the action on $\PP^1$. Therefore $\Psi$ is a trivial bundle over
  $Z\cong \mathcal{Q}^{n-1}$, and $X\cong \PP^1\times \mathcal{Q}^{n-1}$.
\end{proof}

\subsection{Conclusion of the proof of Theorem \ref{bw3classification}}%
\label{conclusion_bw3_classification}
Now fitting together the results of the above subsections, we are able
to prove the classification theorem for bandwidth 3 varieties.

\begin{proof} [Proof of Theorem \ref{bw3classification}] If $n=2$ we reach the claim by Lemma \ref{BW3_surfaces}. Hence, from now on we consider the case $n\geq 3$. We first
  assume that $\rho_{X}\geq 2$. By Lemma \ref{tau>n-3}, we know that
  $\tau\geq n-2$. Moreover, Remark \ref{tau_integer} and Theorem
  \ref{ionescu} imply that $\tau\in \{n-2,n-1,n\}$. If $\tau=n$ then
  by the discussion at the beginning of Subsection \ref{tau>=n} and
  Lemma \ref{scroll_onP1} we get $(1)$.

  Assume that $\tau<n$. We first show that if $\rho_X\geq 2$ and the
  adjunction morphism $\phi_{\tau}$ is the contraction to a point,
  then $(X,L)=(\PP^1\times \PP^1\times\PP^1,\cO(1,1,1))$.  Indeed, if
  $\phi_{\tau}$ is the contraction to a point, applying Lemma
  \ref{bandwidth3-intersection} we deduce that $\tau=\frac{2}{3}n$. We
  analyze what happens for $\tau=\frac{2}{3}n=n-1$, and
  $\tau=\frac{2}{3}n=n-2$. In the first case, applying calculations
  from Remark \ref{isolated_points}, we see that $\rho_X=3$. Moreover,
  by Proposition \ref{tau-first_estimate} (3), (4) $X$ is a Fano
  3-fold and from Theorem \ref{ionescu} (2) it has index 2, then we
  get $(X,L)=(\PP^1\times \PP^1\times\PP^1,\cO(1,1,1))$. The fixed
  point locus of the $\CC^{*}$ action is given by $8$ isolated points
  as described in Example \ref{ex-quad_bdl}.

  Now, we prove that the case $\rho_X\geq 2$ and
  $\tau=\frac{2}{3}n=n-2$ is not possible. If this happens,
  \cite[Theorem B]{WIS2} implies that $(X,L)=(\PP^3\times \PP^3,
  \cO(1,1))$. Then $L=L_1\otimes L_2$ where $L_i$ are the pullback of
  $\cO(1)$ via projections on each of the factors $p_i\colon
  \PP^3\times \PP^3\to \PP^3$. Each $L_i$ is nef and nontrivial,
  therefore, since by our assumption the bandwidth of $L$ is $3$, then
  one of them, say $L_1$ has bandwidth $1$. The contraction $p_1$ is
  equivariant and thus the resulting action of $\CC^*$ on $(\PP^3,
  \cO(1))$ has bandwidth $1$, and by Corollary
  \ref{cor_extremal_points} it has two pointed ends, a
  contradiction.

  Hence, for $n\geq 4$ we may assume that either $\rho_X=1$ or
  $\phi_{\tau}$ is not the contraction to a point. In the latter case, 
  applying Lemma \ref{tau>n-2}, Lemma \ref{dimX'=1}, and Lemma
  \ref{quadric-bdl-str} we obtain $(2)$.  If $\rho_X=1$,
  using Proposition \ref{tau-first_estimate} $(2)$, $(3)$ and Remark
  \ref{isolated_points} we obtain part $(3)$ of the statement, hence
  the claim.
\end{proof}

\begin{rem}\label{rem-normal_of_fixed_pt} Let us focus on Theorem \ref{bw3classification} (1), (2), and keep the notation used in the proof of that theorem. Take the corresponding adjunction morphism $\phi_{\tau}\colon X\to \PP^1$. Since $\phi_{\tau}$ is $\CC^{*}$ equivariant (see Proposition \ref{descending_action}), the fixed locus $X^{\CC^{*}}$ will be contained in the inverse image of the fixed locus of the $\CC^{*}$ action on $\PP^1$. In the scroll case, in the proof of Lemma \ref{scroll_onP1}, we have shown that the fixed point components $Y_1\iso Y_2\iso \PP^{n-2}\subset F\iso \PP^{n-2}$, with $F$ being the fiber of $\phi_{\tau}$. Therefore, we get $\cN^+(Y_1)=\cO(1)$,
  $\cN^-(Y_1)=\cO$, $\cN^+(Y_2)=\cO$, and 
  $\cN^-(Y_2)=\cO(1)$. In the quadric bundle case, in Lemma \ref{quadric-bdl-str}, we proved that $\cY_i=\{pt\}\sqcup Y_i^2$ for $i=1,2$, where $Y_1^2\iso Y^2_2\iso\mathcal{Q}^{n-3}\subset \widetilde{F}\iso \mathcal{Q}^{n-2}$, with $\widetilde{F}$ being the fiber of $\phi_{\tau}$. Hence, one has
  $\cN^+(Y_1^2)=\cO(1)$, $\cN^-(Y_1^2)=\cO(1)\oplus\cO$, $\cN^+(Y^2_2)=\cO(1)\oplus\cO$, and $\cN^-(Y^2_2)=\cO(1)$.
\end{rem}

\section{Contact manifolds}\label{sect-contact-mnflds}

\subsection{Contact manifolds of dimension 11 and 13}\label{contact11-13}

In this section, $X_\sigma$ is a contact variety of dimension $2n+1$ with
$L_\sigma$ an ample line bundle on it, and $\Pic X_\sigma\iso \ZZ
L_\sigma$. By definition, $L_\sigma$ is the cokernel of the contact distribution
$F_\sigma\rightarrow TX_\sigma$ with a rank $2n$ vector subbundle
$F_\sigma\subset TX_\sigma$, and $\sigma\in\HH^0(X_\sigma,\Omega X_\sigma\otimes
L_\sigma)$ such that $d\sigma$ defines a nowhere degenerate pairing
$F_\sigma\times F_\sigma\rightarrow L_\sigma$. In particular,
$-K_{X_\sigma}=(n+1)L_\sigma$.

Contact manifolds appear in the context of quaternion-K\"ahler
manifolds and LeBrun--Salamon conjecture in differential geometry,
which asserts that every positive quaternion-K\"ahler manifold is a
Wolf space. The algebro-geometric version of LeBrun-Salamon conjecture
predicts that every Fano contact manifold is rational homogeneous and,
in fact, isomorphic to the adjoint variety of a simple group, that is
the closed orbit in the projectivisation of the adjoint representation
of this simple group. The contact manifold coming from a
quaternion-K\"ahler manifold admits K\"ahler-Einstein metric, so that in the
differential-geometric context is not restrictive to assume that the group of its
contact automorphisms is reductive.

Let us recall that the case when $\Pic X_{\sigma}\ne \ZZ L_\sigma$ is
known; in such a case $(X_\sigma,L_\sigma)= (\mathbb{P}(TY), \cO_{\mathbb{P}(TY)}(1))$ with $Y$ a projective manifold of dimension $n+1$; see \cite[Corollary 4.2]{L-S}, \cite[Theorem 1.1]{KPSW}, \cite[Corollary 4]{Demailly}. Also the case in which $L_\sigma$ have
sufficiently many sections is known, see \cite[Theorem 0.1]{Beauville}. For $\dim
X_\sigma\leq 9$ we have the following theorem, we refer to \cite[Theorem 1]{DRUEL}, \cite[Theorem 1.2]{B_W}.
\begin{thm}\label{contact_dimleq9}
  Let $(X_\sigma,L_\sigma)$ be a contact Fano manifold of dimension $\leq 9$
  whose group of contact automorphisms $G$ is reductive.  Then $G$ is simple and
  $X_\sigma$ is the closed orbit in the projectivisation of the adjoint
  representation of $G$.
\end{thm}

In \cite{B_W} the proof of the above theorem for $\dim X_\sigma=7, 9$ is based
on the analysis of the action of the maximal torus $\widehat{H}$ in the group
$G$ of contact automorphisms of $X_\sigma$.  The torus $\widehat{H}$ is of rank
$\geq 2$ by a result of Salamon (see \cite[Theorem 7.5]{SALAMON}) reproved in
\cite[Theorem 6.1]{B_W} in the contact case.

In Theorem \ref{contact_dimleq13} we will follow the strategy adopted in \cite{B_W} to extend  the above result. To this end, before recalling the main idea of \cite{B_W}, we briefly remind some preliminaries.

For any manifold $X$ with an ample line bundle $L$ and an almost
faithful action of a torus $H$, one analyses data in the
lattice $M$ of characters of $H$. We recall, see
\cite[$\S$2.1]{B_W}, that a linearization $\mu$ of $L$ defines the
polytope of fixed point $\Delta(L):=\Delta(X,L,H,\mu)$ that
is the convex hull in $M_{\RR}$ of the weights $\mu(Y_i)\in M$ with
which $H$ acts on the fiber of $L$ over each point in a
fixed component $Y_i\subset X^{H}$. Moreover, such a
linearization gives also the polytope of sections
$\Gamma(L):=\Gamma(X,L,H,\mu)$ which is the convex hull in
$M_{\RR}$ of the characters (eigenvalues) of the action of
$H$ on $H^0(X,L)$.

Fixed point components in $\cY$ are represented by points in $M$
together with vectors representing the weights of the action of
$H$ on their conormal bundle; for each $Y\in\cY$ the set of
these (possibly multiple) weights is called the \textit{compass} and denoted by
$\cC(Y, X, H)$ or simply by $\cC(Y)$. We refer to \cite[$\S$2.3]{B_W} for details about the
compass. In the contact case, because of the pairing coming from the
contact form, the vectors in the compass satisfy associated symmetry
(see \cite[Lemma 4.1]{B_W}).

\begin{defi}\label{def-grid-data}
  Given a polarized pair $(X,L)$ with an action of an algebraic torus
  $H$ and linearization $\mu$, we define the {\em grid data}
  of the quadruple $(X,L,H,\mu)$ as follows:
  \begin{enumerate}[leftmargin=*]
  \item the isomorphism classes of connected fixed point components
    $Y_i$, for $X^{H}=\bigsqcup_{i\in I}Y_i$ together with the
    fixed point weight map $$\mu: \cY=\{Y_i: i\in I\}\ra
    M=\Hom(H,\CC^*);$$
    \item the compasses $\cC(Y_i)$ for every $Y_i\subset X^{H}$; and the isomorphism classes of the splitting of the normal bundle
    $$\cN_{Y_i/X}=\bigoplus \cN^{-\nu(Y_i)}(Y_i)$$ where
    $\nu(Y_i)\in \cC(Y_i)$ and
    $\cN^{-\nu(Y_i)}(Y_i)$ are the eigen-subbundles of the
    respective weights.
  \end{enumerate}
\end{defi}

The localized version of Riemann-Roch theorem asserts that the Euler characteristic of $L$,
$\chi^{H}(X,L)$ as a function graded in $M$
depends only on the grid data under certain assumptions, see \cite[Theorem A.1]{B_W}.

The proof of Theorem \ref{contact_dimleq9} for $\dim X_\sigma=7,
9$ goes along the following steps:

\begin{enumerate}[leftmargin=*]\setcounter{enumi}{-1}
\item Prove that there exists a nontrivial action of a (reductive)
  group $G$ with a maximal torus $\widehat{H}$ of rank $r$ on
  $X_\sigma$; it is enough to show that $h^0(X_\sigma,L_\sigma)>0$, see
  \cite{SALAMON} and \cite[Theorem 6.1]{B_W}.
\item Prove that
  $\Delta(X_\sigma,L_\sigma,\widehat{H},\mu)=\Gamma(X_\sigma,L_\sigma,\widehat{H},\mu)$,
  and the vertices of this polytope are associated to isolated fixed
  point components \cite[Lemma 4.7]{B_W}.
\item Prove that $\Gamma(X_\sigma,L_\sigma,\widehat{H},\mu)$ is
  associated to the adjoint representation of the group $G$
  \cite[Lemma 4.5]{B_W} and therefore $G$ is semisimple \cite[Lemma
  4.6]{B_W}.
\item Prove that $G$ is simple and therefore
  $\Delta(L_\sigma)=\Gamma(L_\sigma)$ is the root polytope of $G$ in
  the lattice of weights of $G$ (see \cite[Proposition 4.8]{B_W}).
\item Examine, case by case, root polytopes of simple groups and
  eliminate the ones which are not associated to the action on the
  adjoint contact variety (see \cite[$\S$5]{B_W}).
\item Once it is shown that the grid data of the quadruple $(X_\sigma,L_\sigma,
  \widehat{H},\mu)$ are the same as in the adjoint contact variety
  case, one can conclude that $X_\sigma$ is actually the
  adjoint variety by \cite[Proposition 2.23]{B_W}.
\end{enumerate}

We note that the starting point, that is step (0), is essential to
launch the whole argument. On the other hand, steps (2), (3) and (5) in
this line of argument are fairly general.  Step (1) depends on a general
lemma about existence of sections of an ample line bundle $L_Y$ on a
arbitrary Fano manifold $Y$ such that $\Pic Y=\ZZ L_Y$, and $\dim Y =
n-r+1$. In \cite{B_W}, a well known fact for Fano 3-folds is used. In
what follows, we present a generalization of this result for Fano
4-folds and 5-folds, that is Lemma \ref{anticanonical_dim} (see also \cite[Corollary 1.3]{Smiech}).

The results of step (4) are summarized in \cite[Theorem
5.3]{B_W}. If $\dim X_\sigma\leq 13$ and $r\geq 2$ then
that theorem can be improved by analysing the case of the action of a
simple group of type $A_2$ or $G_2$ on $X_\sigma$. This is done in
Subsection \ref{$SL_3$ action on contact manifolds}. The classification
of bandwidth 3 manifolds given by Theorem \ref{bw3classification} is
the key ingredient in this argument.

As result we obtain the following:
\begin{thm}\label{contact_dimleq13}
  Let $(X_\sigma,L_\sigma)$ be a polarized pair, with $X_\sigma$
  contact Fano manifold of dimension $\leq 13$, and $\Pic X_\sigma=\ZZ
  L_\sigma$. Assume that the group of contact automorphisms $G$ is
  reductive of rank $\geq 2$ (the latter is true if
  e.g.~$h^0(X_\sigma,L_\sigma)> 3$).  Then $X_\sigma$ is a rational
  homogeneous variety, and in particular:
  \begin{enumerate}[leftmargin=*]
  \item if $\dim X_\sigma=11$ then $X_\sigma$ is the closed orbit in
    the projectivisation of the adjoint representation of $SO_9$;
  \item if $\dim X_\sigma=13$ then $X_\sigma$ is the closed orbit
    in the projectivisation of the adjoint representation of
    $SO_{10}$.
  \end{enumerate}
\end{thm}
\begin{proof}
  As noted above, the proof goes along the lines established in
  \cite{B_W}. Namely, using Corollary \ref{cor_Delta=Gamma}, and applying \cite[Proposition
  4.8]{B_W} and \cite[Lemma 4.5]{B_W} we are in situation of
  \cite[Assumptions 5.2]{B_W}. In particular, by that assumptions we recall that the group $G$ is simple. By contradiction, assume that $G$ is of type $A_2$ or $G_2$. In such a case, consider the action of a rank two torus $\widehat{H}\subset G$ on $(X_\sigma, L_\sigma)$. Due to Proposition \ref{fixedpts_n=5-6}, we find out that the grid data of $(X_\sigma, L_{\sigma}, \widehat{H}, \mu)$ coincide with the grid data of $(G(1,\mathcal{Q}^{n+2}),L,H_2,\mu)$ with $H_2$ a rank two torus contained in the maximal torus acting on $G(1,\mathcal{Q}^{n+2})$. From the proof of Propositions \ref{fixedpts_n=5-6}, and \ref{SL3-on-adjoint} we observe that ${L_\sigma}_{\mid Y}\iso L_{\mid Y}$ for every fixed point component $Y$. The equality of the grid data, together with the isomorphism ${L_\sigma}_{\mid Y}\iso L_{\mid Y}$ are equivalent to require that ${L_{\sigma}}_{\mid Y}$ is $\widehat{H}$-equivariantly isomorphic to $L_{\mid Y}$ and that $\cN_{Y/X_{\sigma}}$ is $\widehat{H}$-equivariantly isomorphic to $\cN_{Y/G(1,\mathcal{Q}^{n+2})}$, respectively. This gives an equality of $\widehat{H}$-equivariant Euler characteristics (see \cite[Theorem A.1]{B_W}):
$$\chi^{\widehat{H}}(X_{\sigma}, L_{\sigma}) = \chi^{H_2}(G(1,\mathcal{Q}^{n+2}),L).$$ Then, being $X_{\sigma}$ and $G(1,\mathcal{Q}^{n+2})$ Fano, one has that $H^0(X_{\sigma}, L_{\sigma})$, $H^0(G(1,\mathcal{Q}^{n+2}),L)$ are equal as elements of the representation ring of $\widehat{H}$. 
  Therefore, using again
  Proposition \ref{fixedpts_n=5-6}, if $n=5$ then
  $h^0(X_\sigma,L_\sigma)=\dim{SO_9}=36$, and if $n=6$ then
  $h^0(X_\sigma,L_\sigma)=\dim{SO_{10}}=45$. In both cases, these
  dimensions are bigger than the dimensions of $G_2$ and $A_2$, against \cite[Assumptions 5.2]{B_W} for which $H^{0}(X_\sigma, L_{\sigma})$ can be identified with the Lie algebra of $G$. Thus
  $G$ is neither $G_2$ nor $A_2$. Now, applying \cite[Theorem 5.3]{B_W} we conclude that when $n=5$ one has $(X_\sigma, L_\sigma)\iso (G(1,\mathcal{Q}^{7}),\cO(1))$ with $G=B_4$; while for $n=6$ we get $(X_\sigma, L_\sigma)\iso (G(1,\mathcal{Q}^{8}),\cO(1))$ with $G=D_5$; hence the claim.
\end{proof}
\begin{rem} Notice that the theorem above improves \cite[Theorem 5.3]{B_W}, since when $n=5,6$ the group of the contact automorphisms $G$ cannot be of type neither $G_2$ nor $A_2$. We refer to the recent preprint \cite[Theorem 6.1]{OSCRW2} where, under certain assumptions on the rank of the maximal torus, LeBrun-Salamon conjecture has been proved in arbitrary dimension, dealing also with the cases in which $G$ is of exceptional type. We note that with our approach of downgrading torus action to $\CC^*$ action of bandwidth 3, the assumption that the rank of the group $G$ is at least two is inevitable. The case of rank one group requires understanding bandwidth 4 action of $\CC^*$ on contact manifolds. Finally, as noted above, the fact that $h^0(X_\sigma,L_\sigma)>0$ implies the action of a reductive group of positive rank. On the other hand,  $h^0(X_\sigma,L_\sigma)>3$ implies the action of a reductive group of rank $\geq 2$, because the only rank 1 groups are $\CC^*$ and $PSL(2)$. Although, at present we can not verify either of these inequalities, they seem to be almost equally hard to check, hence the assumption on rank of $G$ being $\geq 2$ is rather harmless.
\end{rem}

\subsection{Dimension of anticanonical systems}
For  the following result see also the recent paper \cite{Smiech} and references therein.
\begin{lemma}\label{anticanonical_dim}
  Let $X$ be a Fano manifold of positive dimension $\leq 5$ with $\Pic
  X = \ZZ L$. Then $h^0(X,L)> 1$.
\end{lemma}
\begin{proof}
  The claim is known if $\dim X\leq 3$, or if $X$ is a Mukai variety of
  index $\dim X -2$.  Hence, it is enough to prove the claim for $\dim
  X = 4$ and $L=-K_X$, and for $\dim X =5$ and $L=-K_X$, or $L=-K_X/2$.
  Moreover, because of Kodaira vanishing, we are left to prove that
  $\chi(X,L)=\sum_i (-1)^ih^i(X,L)> 1$.

  First, assume that $\dim X=4$. We define
  $\chi(t)=\chi(X,t(-K_X))$, and using the Riemann-Roch for 4-folds we
  get
  $$
  \chi(t)= \frac{1}{24}c_1^4\cdot t^4 +\frac{1}{12}c_1^4\cdot t^3 +
  \frac{1}{24}(c_1^2c_2 + c_1^4)\cdot t^2 + \frac{1}{24}c_1^2c_2\cdot
  t \ignore{\left(\frac{7}{360}c_1^2c_2 - \frac{1}{720}c_1^4 +
      \frac{1}{240}c_2^2 + \frac{1}{720}c_1c_3 -
      \frac{1}{720}c_4\right)} +1
  $$
  where $c_i=c_i(TX)$ are the Chern classes, and their intersection is
  evaluated at the fundamental class of $X$. The last coefficient is 1
  because $\chi(0)=1$. Thus we get
  $$\chi(1)=\chi(0) + \frac{1}{6}c_1^4+\frac{1}{12}c_1^2c_2>1.$$
  The inequality follows because $TX$ is stable (see \cite{P-W, HWANG})
  and we can use Bogomolov inequality \cite[Theorem 0.1]{LANGER} to get $$c_1^2c_2\geq\frac{\rk
    TX-1}{2\rk TX}\cdot c_1^4>0.$$

  Now, we consider the case $\dim X=5$.  We use the notation
  of the previous argument; for simplicity we set $d=c_1^5$.  The
  Hilbert polynomial $\chi(t)$ is invariant with respect to the
  Serre's involution $t\mapsto -t-1$.  Using this involution we get
  two possible presentations of its decomposition in $\RR[t]$:
  $$\begin{array}{ll}
    \chi(t)=\frac{d}{120}(t+\frac{1}{2})(t^2+t+a_1)(t^2+t+a_2)&{\rm\ or\ }\\ \\
    \chi(t)=\frac{d}{120}(t+\frac{1}{2})(t^2+b_1t+b_2)((t+1)^2-b_1(t+1)+b_2)
  \end{array}$$
  where $da_1a_2=240$ and $db_2(b_2-b_1+1)=240$, respectively, because $\chi(0)=1$.
  We can compare it with the Riemann-Roch formula:
  $$\begin{array}{ccl}
    \chi(t)&=&\frac{1}{120}c_1^5\cdot t^5 +\frac{1}{48}c_1^5\cdot t^4 +
    \frac{1}{72}(c_1^3c_2 + c_1^5)\cdot t^3 + \frac{1}{48}c_1^2c_2\cdot t^2
    \\[2pt] &&+ \frac{1}{720}\left(-c_1^5+4c_1^2c_2+3c_1c_2^2+c_1^2c_3-c_1c_4 \right)
    \cdot t +1
  \end{array}.$$
So we get respective identities $$c_1^3c_2=\frac{d}{5}\cdot(3a_1+3a_2+1){\rm \ and \ }
  c_1^3c_2=\frac{d}{5}\cdot(6b_2-3b_1^2+3b_1+1).$$
  Using these identities, we verify that
  $$\chi(1)=3\chi(0)+\frac{1}{24}\cdot(c_1^3c_2+c_1^5){\rm \ and\ }
  \chi \left(\frac{1}{2}\right)=2\chi(0)+\frac{1}{96}c_1^3c_2+\frac{1}{384}c_1^5$$
  which, again, by the stability of $TX$ and
  Bogomolov inequality, yields the lemma.
\end{proof}

We now obtain the following result which
  improves \cite[Lemma 4.7]{B_W}.
\begin{cor}\label{cor_Delta=Gamma}
  Let $X_\sigma$ be a contact Fano manifold of dimension $2n+1$ with
  $\Pic X_\sigma=\ZZ L_\sigma$. Suppose that $X_\sigma$ admits an
  almost faithful action of a torus $\widehat{H}$ of rank $r$.  If
  $r\geq n-4$ then
  $\Gamma(X_\sigma,L_\sigma,\widehat{H})=\Delta(X_\sigma,L_\sigma,\widehat{H})$, and every extremal component of $X_{\sigma}^{\widehat{H}}$ is a
  point.
\end{cor}
\begin{proof}
  The proof is the same as the one of \cite[Lemma 4.7]{B_W} but in place
  of \cite[Corollary 3.8]{B_W} we use the respective version
  following from Lemma \ref{anticanonical_dim} of the present paper.
\end{proof}

\subsection{$SL_3$ action on contact manifolds} \label{$SL_3$ action on contact manifolds}

In this subsection we consider the following situation; compare with
\cite[Assumptions 5.1]{B_W}.
\begin{assumpt}\label{A2_on_contact}
  Let $G$ be a simple group of type $A_2$ or $G_2$ with a maximal two
  dimensional torus $\widehat{H}<G$.  Assume that $G$ acts almost faithfully
  via contactomorphisms on a contact manifold $(X_\sigma,L_\sigma)$, with
  $\dim X_\sigma=2n+1$ and $\Pic X_\sigma=\ZZ L_\sigma$.  That is, the morphism $G
  \rightarrow \Aut{(X_\sigma)}$ has finite kernel.  The linearization
  $\mu$ comes from the action of $G$ on $TX_\sigma$.  Assume that all
  extremal fixed points of the action of $\widehat{H}$ on $X_\sigma$ are
  isolated, and the polytope $\Delta(X_\sigma,L_\sigma,\widehat{H},\mu)$ is the root
  polytope $\Delta(G)$ in the lattice $\widehat{M}$ of characters of the torus
  $\widehat{H}$.
\end{assumpt}

The following diagram is copied from \cite[$\S$5.5]{B_W}. We use
the notation coming from that paper.

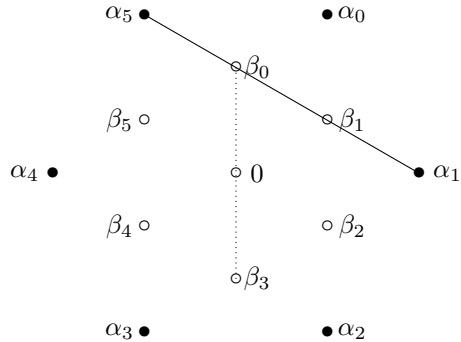
\begin{figure}[h!!]
$$ 
\begin{xy}<40pt,0pt>:
(0,0)*={\circ}="1", (0.2,0)*={0},
(0,1)*={\circ}="b0", (0.2,1)*={\beta_0},
(0.866,0.5)*={\circ}="y", (1.1,0.5)*={\beta_1}, 
(0.866,-0.5)*={\circ}="x", (1.1,-0.5)*={\beta_2},
(0,-1)*={\circ}="b3", (0.2,-1)*={\beta_3},
(-0.866,-0.5)*={\circ}="y^(-1)", (-1.1,-0.5)*={\beta_4},
(-0.866,0.5)*={\circ}="x^(-1)", (-1.1,0.5)*={\beta_5},
(0.866,1.5)*={\bullet}="a0", (1.1,1.5)*={\alpha_0},
(1.732,0)*={\bullet}="a1", (2,0)*={\alpha_1},  
(0.866,-1.5)*={\bullet}="a2", (1.1,-1.5)*={\alpha_2},
(-0.866,-1.5)*={\bullet}="a3", (-1.1,-1.5)*={\alpha_3},
(-1.732,0)*={\bullet}="a4", (-2,0)*={\alpha_4},
(-0.866,1.5)*={\bullet}="a5", (-1.1,1.5)*={\alpha_5},
"a5";"a1" **@{-},"b0";"b3" **@{.},
\end{xy}
$$
\caption{Lattice points corresponding to the action of $\widehat{H}$.\label{exagon_lattice}}
\end{figure}

By $y_{\alpha_i}$ we denote the extremal fixed points of the action of
$\widehat{H}$; by $Y^j_{\beta_i}$ we denote the inner fixed point
components associated to the weight $\beta_i$, while by $Y_0$ we
denote central components associated to the weight 0.  The indices of
$\alpha$'s and $\beta$'s are between 0 and 5; by convention they
are taken modulo 6.

Fix $i\in \{0,\dots,5\}$ and let $H^{\prime}$ be the subtorus corresponding to the projection $\pi_{i}\colon \widehat{M}\to \mathbb{Z}$ which maps $\alpha_{i-1}, \beta_{i},\beta_{i+1}, \alpha_{i+1}$ to $1\in \mathbb{Z}$.  Using that $\rho_{X_\sigma}=1$, and arguing as in the proof of Proposition \ref{onepointend-basic}, we deduce that there exists a unique connected component $X_{i}\subset X_\sigma^{H^{\prime}}$ which contains the extremal fixed points $y_{\alpha_{i}+1}$, $y_{\alpha_{i}-1}$, and all inner components of $X_\sigma^{\widehat{H}}$ associated to $\beta_{i}$ and $\beta_{i+1}$. On
the above diagram, we indicate the (solid) line segment associated to $X_i$. We will use the convention that $y_{\alpha_{i-1}}$ is the sink and
$y_{\alpha_{i+1}}$ is the source of the action of $\CC^*$ on $X_i$. 

\begin{lemma}\label{A2=>bandwidth3}
  Let us keep the above notation, and suppose that Assumptions
  \ref{A2_on_contact} are satisfied. Then the following hold:
\begin{enumerate}[leftmargin=*]
\item $X_i$ is a smooth connected variety of dimension $n-1$ with an
  ample line bundle $L_i:={L_\sigma}_{\mid X_i}$,
\item $X_i$ admits a natural $\CC^*$ action and a
  natural linearization $\mu_i$ of $L_i$,
\item the fixed point components of $X_i^{\CC^{*}}$ are
  $\{y_{\alpha_{i-1}}\}, Y^j_{\beta_i}, Y^j_{\beta_{i+1}},
  \{y_{\alpha_{i+1}}\}$,
\item $(X_i,L_i)$ is a bandwidth 3 variety with two end points, and
  the action of $\CC^*$ is equalized.
\end{enumerate}
\end{lemma}
\begin{proof} By construction, $X_i$ is smooth and connected. Moreover, $X_i$ admits the restricted action of the $1$-dimensional torus $\CC^{*}=\widehat{H}/H^{\prime}$ as required in $(2)$ (see \cite[Lemma 2.10 (2)]{B_W}), and by \cite[Lemma 2.10 (3)]{B_W} one has $X_i^{\CC^{*}}=X_{\sigma}^{\widehat{H}}\cap X_i$; therefore the extremal fixed points of the $\CC^{*}$ action
have weights $\alpha_{i-1}$ and $\alpha_{i+1}$, thus the bandwidth of $(X_i,L_i,\CC^{*})$ is three. We are left to check that the $\CC^{*}$ action is equalized, and that $\dim{X}=n-1$. To this end, let us describe the compasses at the fixed point components. Taking a fixed component $Y\subset X_i^{\CC^{*}}$, applying the definition of the compass, and using the projection $\pi_{i}\colon \widehat{M}\to\ZZ$, one has
\begin{equation}  \label{compass}
\cC(Y,X_i,\CC^{*})=\cC(Y,X_{\sigma}, \widehat{H}) \cap \text{ker}(\pi_i).
\end{equation}
  The description of the compasses for the rank two torus $\widehat{H}$ are obtained following the same proof of
  \cite[Lemma 5.15]{B_W}. In what follows, we use an exponent to denote the occurrence of the corresponding element in the compass. By \cite[Lemma 5.15 (1)]{B_W}, we obtain that:
 \begin{equation} \label{compass_alpha}
 \begin{array}{rl}
 \cC(y_{\alpha_{i-1}}, X_{\sigma}, \widehat{H})=&(\alpha_{i}-\alpha_{i-1}, \alpha_{i-2}-\alpha_{i-1}, -\alpha_{i-1}, {(\beta_{i}-\alpha_{i-1})}^{n-1},\\[2pt] & {(\beta_{i-1}-\alpha_{i-1})}^{n-1}).
 \end{array}
 \end{equation}

Using (\ref{compass}) we deduce that $\cC(y_{\alpha-1},X_i,\CC^{*})=(1^{n-1})$. In a similar way, we obtain that $\cC(y_{\alpha+1},X_i,\CC^{*})=(-1^{n-1})$. This also implies that $\dim X=n-1$, because by definition of the compass at a fixed component $Y$, the elements contained in it must be in number equal to $\dim X-\dim Y$; and if we consider $Y=y_{\alpha-1}$, being the sink a point by assumption, we obtain claim $(1)$.

For an irreducible inner fixed point component $Y_{\beta_i}$ of dimension $d$, the proof of  \cite[Lemma 5.15 (2)]{B_W} allows to compute
\begin{equation} \label{compass_inner}
\begin{array}{rl}
\cC(Y_{\beta_i}, X_{\sigma}, \widehat{H})=&(\alpha_{i}-\beta_{i}, \alpha_{i-1}-\beta_{i},\beta_{i+2}-\beta_{i}, \beta_{i-2}-\beta_{i},\\[2pt] & -\beta_{i}^{d+1}, (\beta_{i+1}-\beta_{i})^{n-d-2}, (\beta_{i-1}-\beta_{i})^{n-d-2});
\end{array}
\end{equation}
 and by (\ref{compass}) we obtain that $\cC(Y_{\beta_i},X_i,\CC^{*})=(1^{n-d-2},-1)$. Repeating the same procedure for the other inner fixed point components, we may conclude that the $\CC^{*}$ action on $X_i$ is equalized, and the statement follows. 
\end{proof}

\begin{lemma}\label{normal_Y_beta} Let us keep the above notation, and
  suppose that Assumptions \ref{A2_on_contact} hold. Then
$$\cN_{Y^j_{\beta_i}/X_i}^-\iso \left( \cN_{Y^j_{\beta_i}/X_{i-1}}^+
\right)^*\otimes L_{|Y^j_{\beta_i}}$$
\end{lemma}
\begin{proof}
  The pairing $d\sigma: F_\sigma\times F_\sigma\rightarrow L_\sigma$
  is invariant with respect to the action of $\widehat{H}$.  Hence it
  determines the pairing on the normal of the eigencomponents of the
  normal to any fixed point component.
\end{proof}
\begin{cor}\label{contact-no_scroll}
  The variety $(X_i,L_i)$ described in Lemma \ref{A2=>bandwidth3} is
  not a scroll over $\PP^1$ described in case (1) of
  Theorem \ref{bw3classification}.
\end{cor}
\begin{proof}
  In the scroll case $\cN_{Y^j_{\beta_i}/X_i}^-\iso
  \cN_{Y^j_{\beta_i}/X_{i-1}}^+\iso \cO$, see Remark
  \ref{rem-normal_of_fixed_pt}, which contradicts Lemma
  \ref{normal_Y_beta}.
\end{proof}

\begin{lemma} \label{central_components} Suppose that Assumptions \ref{A2_on_contact} are satisfied, and keep the same notation there introduced. Then:
\begin{enumerate}[leftmargin=*]
\item if $n=5$ there are no central components;
\item if $n=6$ one has $Y_0=\PP^1\sqcup \PP^1$, and the compass $\cC(\PP^1, X_\sigma, \widehat{H})$ is given by all the vectors $\pm \beta_i$ for $i=0,1,2$, where each element occurs with multiplicity two.
\end{enumerate}
\end{lemma}

\begin{proof} We first show that if $Y_{0,k}$ is an irreducible central component, then the elements $\alpha_i\not \in \cC(Y_{0,k},X_\sigma, \widehat{H})$. Assume by contradiction that one of these elements, say $\alpha_0$, belongs to the compass. Consider a subtorus $H_1\subset \widehat{H}$ corresponding to a projection $\pi\colon \widehat{M}\to \ZZ$ sending $\alpha_0$ to $0$. Take a variety $Z\subset X_\sigma^{H_1}$ which contains $Y_{0,k}$ and the extremal fixed points $y_{\alpha_0}$, $y_{\alpha_3}$. Applying \cite[Lemma 2.10 (2)]{B_W} such a variety $Z$ admits the action of ${\CC^{*}}=\widehat{H}/H_1$, with fixed locus $Z^{\CC^{*}}=y_{\alpha_0}\sqcup Y_{0,k}\sqcup y_{\alpha_3}$. Moreover, by Corollary \cite[Corollary 4.4]{B_W} the variety $Z$ is contact. Replacing $X_i$ with $Z$ and $\pi_i$ with $\pi$ in the formula (\ref{compass}), and using (\ref{compass_alpha}) we get $\cC(y_{\alpha_0}, Z, \CC^{*})=(-\alpha_0)$, therefore $\dim Z=1$. We then conclude $Z\iso \PP^1$, so that $Y_{0,k}=\emptyset$, a contradiction. Hence, if $Y_{0,k}\neq \emptyset$, applying \cite[Corollary 2.14]{B_W}, it follows that the only elements which belong to $\cC(Y_{0,k},X_\sigma, \widehat{H})$ are among the $\beta_i$'s. On the other hand, using Lemma \ref{A2=>bandwidth3}, Corollary \ref{contact-no_scroll}, and Theorem \ref{bw3classification} we deduce that if $n=5$ one has $Y_{\beta_i}=\{pt\} \sqcup \PP^1$; if $n=6$ we have $Y_{\beta_i}=\{pt\} \sqcup \PP^1\times \PP^1$.

By the above argument, if $Y_{0,k}\neq \emptyset$, we may assume that the element $\beta_0$ belongs to $\cC(Y_{0,k},X_\sigma, \widehat{H})$. Being $Y_{0,k}$ contact (see \cite[Corollary 4.4]{B_W}), its compass is symmetric (see \cite[Lemma 4.1]{B_W}), therefore also $\beta_3$ has to belong to the compass with the same multiplicity of $\beta_0$.

Now, take a subtorus associated to the projection of the lattice $\widehat{M}$ along the dotted line segment in Figure \ref{exagon_lattice}. Applying the reduction procedure explained above, we find a contact variety $Z^{\prime}\subset X_\sigma^{\widehat{H}}$ which contains an irreducible fixed component of $Y_{\beta_0}$, and of $Y_{\beta_3}$; these fixed components will be respectively the sink and the source of a bandwidth two $\CC^{*}$ action on $Z^{\prime}$. We may exclude the case in which the sink (or the source) of such an action is the isolated fixed point in $Y_{\beta_0}$ (resp. in $Y_{\beta_3}$); indeed in such a case, arguing as above, we would again have $Y_{0,k}= \emptyset$. Hence, the two extremal points of the $\CC^{*}$ action on $Z^{\prime}$ are both isomorphic to $\mathcal{Q}^{n-4}$.

Using (\ref{compass}) and (\ref{compass_inner}), we get $\cC(Y_{\beta_0}, Z^{\prime}, \CC^{*})=(-\beta_0^{n-3})$; therefore $\dim Z^{\prime}=2n-7$; so that if $n=5$ one has $Z^{\prime}\iso \PP^{3}$; if $n=6$ we get $Z^{\prime}\iso\PP(T\PP^{3})$.

In the first case, since the extremal components are isomorphic to $\PP^{1}$ with the restriction of $L_\sigma$ being $\cO(2)$ (see Theorem  \ref{bw3classification} (2)), we are considering a $\CC^{*}$ on $(\PP^{3}, \cO(2))$ having weights $(0,0,1,1)$, and such an action does not have central components. 

  If $n=6$, we consider the corresponding polytope of fixed points
  of the big torus in $SL_4$ acting on $Z^{\prime}$ (see the picture of
  \cite[Exercise 15.10]{FULTON_HARRIS}), and by a downgrading associated
  to the projection along one of the faces of the cube in which the
  polytope is inscribed we get $Y_{0}=\PP^{1}\sqcup \PP^{1}$. We have already observed that the compass at the central components is symmetric. Using this fact, and being $\dim Z^{\prime}=5$, one has $\cC(\PP^{1},Z^{\prime},\CC^{*})=(\beta_0,\beta_0,\beta_3, \beta_3)$. Since $\dim X_\sigma=13$, the compass $\cC(\PP^{1},X_\sigma,\widehat{H})$ contains $12$ elements counted with their multiplicity. Therefore, repeating the same argument with the weights $\beta_1$ and $\beta_2$, we obtain the claim.
\end{proof}

\begin{prop}\label{fixedpts_n=5-6}
  Suppose that Assumptions \ref{A2_on_contact} are satisfied. Then
  for $n=5, 6$ the grid data of the quadruple $(X_\sigma,L_\sigma,
  \widehat{H},\mu)$ is the same as for the quadruple
  $(G(1,\mathcal{Q}^{n+2}),L,H_2,\mu)$ from Proposition \ref{SL3-on-adjoint}.
\end{prop}
\begin{proof}
  We need to compare the information about the fixed point components
  for $X_\sigma$ with the result obtained in Proposition
  \ref{SL3-on-adjoint}. The information about inner components
  $Y_{\beta_i}$ are given by downgrading to subvarieties
  $X_i$. Indeed, because of Lemma \ref{A2=>bandwidth3} and Corollary
  \ref{contact-no_scroll}, we are in the situation of case (2) of Theorem
  \ref{bw3classification}; see also Remark
  \ref{rem-normal_of_fixed_pt}. Therefore, the isomorphism classes of
  the components and their normal subbundles in $X_i$ are uniquely
  determined. We now observe that the same holds for the central components. To this end, we apply Lemma \ref{central_components} to $(X_\sigma,L_\sigma)$ and to $(G(1,\mathcal{Q}^{n+2}),L)$. 
In particular, when $n=6$, by the proof of the same lemma we recall that the central component $Y_0=\PP^{1}\sqcup \PP^{1}$ is contained in three distinct varieties $Z_{j}$ admitting a bandwidth two $\CC^{*}$ action whose sink and source are respectively associated to the weights $\beta_0$ and $\beta_3$; $\beta_1$ and $\beta_4$; $\beta_2$ and $\beta_5$.  Since the normal bundle of each copy of $\PP^{1}$ in $Z_{j}$ is uniquely determined, also its decomposition
  according to the weights of the $\CC^{*}$ action will be uniquely determined, and the statement follows.
\end{proof}

\appendix
\section{Embedding $SL_3$ into classical linear groups}\label{appendix}
In this appendix we summarize information about the structure of the adjoint
variety $X_{adj}$ of a classical simple algebraic group $G$ from the viewpoint
of a linear embedding of the group $SL_3$ into the group in question. Let us
recall that $X_{adj}$ is the closed orbit in the projectivisation of the adjoint
representation of $G$.

We focus on the case of a linear embedding $SL_3\hookrightarrow SO_m$ which
yields the action of $SL_3$ on the adjoint variety of $SO_m$. The results of
this section are stated in Proposition \ref{SL3-on-adjoint}. The conclusion is
that the associated bandwidth 3 variety which we described in Subsection
\ref{$SL_3$ action on contact manifolds} (see Lemma \ref{A2=>bandwidth3}) yields the case (2) in Theorem
\ref{bw3classification}.

First, let us recall that the root systems of $SL_4$ and $SO_6$ coincide and
their adjoint variety is $\PP(T\PP^3)$. We consider a natural embedding
$SL_3\hookrightarrow SL_4$ which comes with the linear embedding of the standard
representation $W_3$ of $SL_3$ into the standard representation of $SL_4$, that is
$W_4=W_3\oplus\CC$, as representation of $SL_3$, where  $\CC$ denotes the
trivial representation of $SL_3$. The adjoint representation $adj(SL_4)$ of
$SL_4$ is an irreducible summand of $W_4\oplus W_4^*=adj(SL_4)\oplus\CC$; thus as a representation of $SL_3$ it decomposes as  $$adj(SL_4)=W_3\oplus
adj(SL_3)\oplus W_3^*\oplus\CC$$

On the other hand, an embedding $SL_3\hookrightarrow SO_6$ is defined as follows
$$SL_3\ni A\longrightarrow \widehat{A}=
\left(\begin{array}{cc}0&(A^\intercal)^{-1}\\ A&0\end{array}\right)\in SO_6$$
Where $SO_6$ is understood as the group of matrices preserving the form
$\left(\begin{array}{cc}0&I\\ I&0\end{array}\right)$, or a quadric in $\PP^{5}$
given by the equation $x_1y_1+\cdots +x_3y_3=0$. If $V_6$ is the standard
representation of $SO_6$, then as a representation of $SL_3\hookrightarrow
SO_6$ it decomposes as $W_3\oplus W_3^*$. The second exterior power $\bigwedge^2
V_6$ is the adjoint representation of $SO_{6}$ and, as the representation of
$SL_3$, it can be written again as $$\bigwedge^2(W_3\oplus W_3^*)= \bigwedge^2
W_3\oplus W_3\otimes W_3^*\oplus\bigwedge^2W_3^*\oplus\CC=W_3\oplus
adj(SL_3)\oplus W_3^*\oplus\CC$$

In terms of the action of the Cartan torus $H_3$ of both $SL_4$ and $SO_6$, the
fixed points of the action on the adjoint variety via the standard linearization
map are mapped to roots in the associated rank three lattice of weights $M_3$.
As points in the space $M_3\otimes\RR$ they are vertices of a cuboctahedron
(rectified cube). The embedding of $SL_3$ in each of these groups yields an
embedding of Cartan tori $H_2\hookrightarrow H_3$, with $H_2$ the 2-dimensional
torus contained in $SL_3$. Thus we get the projection of the corresponding
lattices of weights $M_3\rightarrow M_2$. The diagram below describes the
projection of the roots visualized as a projection of the cuboctahedron. The front
and the rear faces are associated to representation $W_3$ and $W_3^*$, while the
hexagonal cross-section is associated to the representation $adj(SL_3)$.

\begin{center}
\includegraphics[width=380pt]{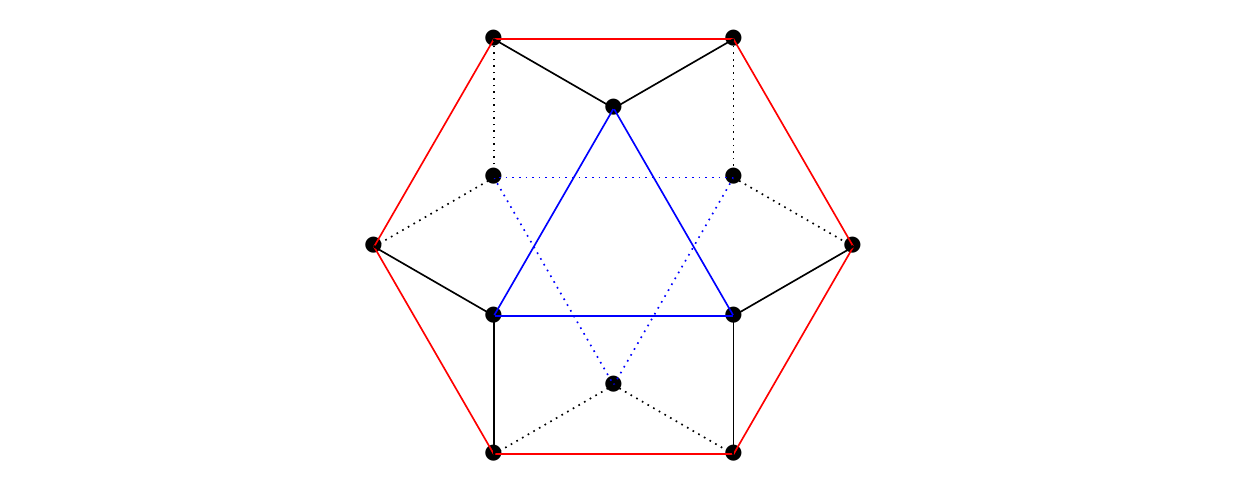}
\end{center}

By \cite[Lemma 4.5]{B_W}, the vertices of the root polytopes are associated to
the fixed points of the action of the Cartan torus on $X_{adj}$ under the fixed
point weight map for the standard linearization. Thus, it follows that the dots
$\bullet$ are also the images of these fixed points.

We can write the above description directly in terms of the coordinates of the
action of the torus $H_2\subset SL_3$.  Its associated lattice of characters is 
$M_2=\bigoplus_{i=1}^3 \ZZ e_i/\sum e_i$. We choose the coordinates $(x_i,y_i)$
on $V_6$ so that the weights of the action of $H_2$ on $V_6=W_3\oplus W_3^*$ are
$e_i$ on $x_i$ and $-e_i$ on $y_i$. There are six fixed points of the action of
$H$ on the quadric $\mathcal{Q}^4=\{\sum x_iy_i=0\}\subset\PP(V_6)$, each associated to
the weight $\pm e_i$.

Recall that $X_{adj}$ for $SO_6$ parametrizes the lines contained in the quadric. If
$\alpha\wedge\beta$ denotes the line for which only coordinates $\alpha$ and
$\beta$ do not vanish, then the $H_2$ invariant lines contained in the quadric $\sum
x_iy_i=0$ are either of the two types:
  \begin{itemize}
    \item $x_i\wedge y_j$ for $i\ne j$
    \item $x_i\wedge x_j$, or $y_i\wedge y_j$ for $i\ne j$.
  \end{itemize}
There are six lines of each type; the weight of the action of $H_2$ on the line of
the first type is $e_i-e_j$, while on the latter type is $\pm e_i$.  These are the
fixed points of the action described above as projection of vertices of
the cuboctahedron.

When $m\geq 7$ we take a natural inclusion  $$SL_3\hookrightarrow SO_6\times
SO_{m-6}\hookrightarrow SO_m$$ for a suitable decomposition of the standard $SO_m$
representation $V_m=V_6\oplus V_{m-6}$. As before, we decompose the resulting
$SL_3$ representation:
  $$\bigwedge^2 V_m = W_3\oplus adj(SL_3) \oplus \CC \oplus W_3^* \oplus
  \left(W_3\otimes V_{m-6}\right) \oplus \left(W_3^*\otimes V_{m-6}\right)
  \oplus \bigwedge^2 V_{m-6}$$
where the representation $V_{m-6}$ is trivial as $SL_3$ representation.

We extend the preceding discussion to the $SO_m$ invariant quadric
$\widehat{\mathcal{Q}}\subset\PP(V_m)$ such that $\widehat{\mathcal{Q}}\cap\PP(V_6)
=\{\sum x_iy_i=0\}=\mathcal{Q}$. That is, for suitably chosen coordinates $(z_i)$ in
$V_{m-6}$ we have
$$\widehat{\mathcal{Q}}=\{x_1y_1+x_2y_2+x_3y_3+z_1^2+\cdots+ z_{m-6}^2=0\}$$
By $\mathcal{Q}^\perp$ we denote the intersection $\widehat{\mathcal{Q}}$ with
$\PP(V_{m-6})=\{x_i=y_j=0\}$. That is $\mathcal{Q}^\perp=\{z_1^2+\cdots +z_{m-6}^2=0\}$
is a quadric of dimension $m-8$.
Now, apart of the lines contained in $\mathcal{Q}$, we have extra components in the fixed
point locus of the action of $H_2\subset SL_3$ on the grassmannian of lines on
$\widehat{\mathcal{Q}}$. They are as follows:
\begin{itemize}
\item lines joining fixed points of the action of $H_2$ on $\mathcal{Q}$ with any point
in $\mathcal{Q}^\perp$; they are contained in the subspaces $W_3\otimes V_{m-6}$ and
$W^*_3\otimes V_{m-6}$ in the decomposition of $\bigwedge^2V_{m-6}$ above;
therefore there are six of such components associated to the weights $\pm e_i$;
\item lines contained in $\mathcal{Q}^\perp$ for $m\geq 9$; they are contained in the
subspace $\bigwedge^2V_{m-6}$ in the above decomposition; therefore these
fixed point component(s) are associated to the weight 0.
\end{itemize} We summarize the discussion in the following.
\begin{prop}\label{SL3-on-adjoint}
Assume the situation as above.   The following is the list of the fixed
components of the fixed point locus of the action of the 2-dimensional torus
$H_2\subset SL_3\subset SO_m$ on the Grassmannian of lines in the quadric
$\mathcal{Q}^{m-2}$ denoted by $G(1,\mathcal{Q}^{m-2})$, which is a contact variety of dimension
$2m-7$. Let us denote by $L$ an ample line bundle generating $\Pic
G(1,\mathcal{Q}^{m-2})$.
  \begin{enumerate}
  \item For $m\geq 6$ one point extremal components associated to weights $\pm
    e_i+\pm e_j$, $i\ne j$;
  \item for $m\geq 6$ one point components associated to weights $\pm e_i$;
  \item for $m\geq 8$ additional components associated to weights $\pm
    e_i$ which are quadrics of dimension $m-8$; in particular they are
    \begin{enumerate}
    \item two points for each weight, for $m=8$;
    \item a conic, that is $\PP^1$ with $L_{|\PP^1}\iso\cO(2)$, for $m=9$;
    \item a quadric $\PP^1\times\PP^1$ with restriction of $L$ being
      $\cO(1,1)$, for $m=10$;
    \item a quadric $\mathcal{Q}^{m-8}$ with restriction of $L$ being $\cO(1)$,
      for $m\geq 11$;
    \end{enumerate}
  \item central component(s) for $m\geq 10$ which is the grassmannian
    of lines in the quadric $\mathcal{Q}^{m-8}$, in particular
    \begin{enumerate}
    \item $\PP^1\sqcup\PP^1$ for $m=10$;
    \item an irreducible variety for $m\geq 11$.
    \end{enumerate}
  \end{enumerate}
\end{prop}

A similar discussion can be made in the case of classical linear groups. In the
following table we present adjoint varieties as well as bandwidth 3 varieties
and central components associated to the downgrading of the action to a linear
embedding of $SL_3$. Notation: $G$ is the group, $X_{adj}$ is the adjoint
variety for the group, $\dim X_{adj}=2n+1$. In the spirit of Subsection
\ref{$SL_3$ action on contact manifolds}, $X_i$ is the bandwidth 3 variety associated to
restricting and downgrading of the group action following the embedding
$SL_3\hookrightarrow G$, $\dim X_i=n-1$; moreover $\cY_*$ is the set of
fixed point components in $\cY_1$ or $\cY_2$.  Finally, $Y_0$ is the union of fixed point components associated to the weight 0.\par\bigskip
$$\begin{array}{c|c|c|c|c|c|c}
    n&G&\rk G&X_{adj}&X_i&Y_*&Y_0\\ \hline
    3&SO_7&3&G(1,\mathcal{Q}^5)&\PP^1\times\PP^1&\bullet&\emptyset\\
    4&SO_8&4&G(1,\mathcal{Q}^6)&\PP^1\times\PP^1\times\PP^1&\bullet\sqcup\bullet\sqcup\bullet&\emptyset\\
    5&SO_{9}&4&G(1,\mathcal{Q}^7)&\PP^1\times\mathcal{Q}^3&\bullet\sqcup\PP^1&\emptyset\\
    6&SO_{10}&5&G(1,\mathcal{Q}^8)&\PP^1\times\mathcal{Q}^4&\bullet\sqcup\PP^1\times\PP^1&\PP^1\sqcup\PP^1\\
    7&SO_{11}&5&G(1,\mathcal{Q}^9)&\PP^1\times\mathcal{Q}^5&\bullet\sqcup\mathcal{Q}^3&\PP^3\\
    \geq8&SO_{n+4}&\lfloor\frac{n}{2}\rfloor&G(1,\mathcal{Q}^{n+2})&\PP^1\times\mathcal{Q}^{n-2}&\bullet\sqcup\mathcal{Q}^{n-4}&G(1,\mathcal{Q}^{n-4})\\
    \geq3&Sp_{2n+2}&n+1&\PP^{2n+1}&\PP^1&\emptyset&\PP^{2n-5}\\
    \geq3&SL_{n+2}&n+1&\PP(T\PP^{n+1})&\PP^{n-2}\sqcup\PP^{n-2}&\PP^{n-3}&\PP(T\PP^{n-2})\\
  \end{array}$$

\bibliography{biblio}

\begin{thebibliography}{10}

\bibitem{Beauville}
A.~Beauville.
\newblock Fano contact manifolds and nilpotent orbits.
\newblock {\em Comment. Math. Helv.}, 73(4):566--583, 1998.

\bibitem{B_S}
M.~Beltrametti and A.~J. Sommese.
\newblock {\em The adjunction theory of complex projective varieties},
  volume~16.
\newblock Walter de Gruyter, 1995.

\bibitem{BB}
A.~Bia{\l}ynicki-Birula.
\newblock Some theorems on actions of algebraic groups.
\newblock {\em Ann. of {m}ath.}, 98:480--497, 1973.

\bibitem{BRION}
M.~Brion.
\newblock Linearization of algebraic group actions.
\newblock Lecture notes. Available at
  https://www-fourier.ujf-grenoble.fr/$\sim$mbrion/lin.pdf/.

\bibitem{B_W}
J.~Buczy{\'n}ski, J.~A. Wi{\'s}niewski, and A.~Weber.
\newblock Algebraic torus actions on contact manifolds.
\newblock {\em arXiv preprint:1802.05002, to appear in the Journal of
  Differential Geometry}, 2018.

\bibitem{CARRELL}
J.~B. Carrell.
\newblock Torus actions and cohomology.
\newblock In {\em Algebraic quotients. Torus actions and cohomology. The
  adjoint representation and the adjoint action}, volume 131, pages 83--158.
  Springer, 2002.

\bibitem{Cox}
D.~Cox, J.~Little, and H.~K. Schenck.
\newblock {\em Toric varieties}.
\newblock American Mathematical Soc., 2011.

\bibitem{Demailly}
J.~P. Demailly.
\newblock On the {F}robenius integrability of certain holomorphic {$p$}-forms.
\newblock In {\em Complex geometry ({G}\"ottingen, 2000)}, pages 93--98.
  Springer, Berlin, 2002.

\bibitem{DRUEL}
S.~Druel.
\newblock {Structures de contact sur les vari\'et\'es alg\'ebriques de
  dimension 5.}
\newblock {\em {C. R. Acad. Sci., Paris, S\'er. I, Math.}}, 327(4):365--368,
  1998.

\bibitem{EDIDIN}
D.~Edidin and W.~Graham.
\newblock Localization in equivariant intersection theory and the {B}ott
  residue formula.
\newblock {\em American Journal of Mathematics}, pages 619--636, 1998.

\bibitem{FUJITA}
T.~Fujita.
\newblock On the structure of polarized manifolds with total deficiency one,
  {I},{II} and {III}.
\newblock {\em Journal of the Mathematical Society of Japan}, 32, 33, and
  36:709--725, 415--434, and 75--89, 1980, 1981, 1984.

\bibitem{FUJITA-Sendai}
T.~Fujita.
\newblock On polarized manifolds whose adjoint bundles are not semipositive.
\newblock In {\em Algebraic Geometry, {S}endai, 1985}, volume~10 of {\em Adv.
  Stud. Pure Math.}, pages 167--178. North-Holland, Amsterdam, 1987.

\bibitem{FULTON_HARRIS}
W.~Fulton and J.~Harris.
\newblock {\em Representation theory: a first course}, volume 129.
\newblock Springer-Verlag, 1991.

\bibitem{Smiech}
A.~H\"oring and R.~{\'S}miech.
\newblock Anticanonical system of {F}ano fivefolds.
\newblock {\em Mathematische Nachrichten}, 293(1):115--119, 2020.

\bibitem{HWANG}
J.-M. Hwang.
\newblock Stability of tangent bundles of low-dimensional {F}ano manifolds with
  {P}icard number {$1$}.
\newblock {\em Math. Ann.}, 312(4):599--606, 1998.

\bibitem{IONESCU}
P.~Ionescu.
\newblock Generalized adjunction and applications.
\newblock {\em Math. Proc. Cambridge Phil. Soc.}, 99(3):457--472, 1986.

\bibitem{ISK}
V.~A. Iskovskikh.
\newblock Fano 3-folds. {I}.
\newblock {\em Math. USSR Izv.}, 11:485--527, 1977.

\bibitem{IVERSEN}
B.~Iversen.
\newblock A fixed point formula for action of tori on algebraic varieties.
\newblock {\em Invent. Math.}, 16(3):229--236, 1972.

\bibitem{J-S}
J.~Jelisiejew and {\L}.~Sienkiewicz.
\newblock {B}ia{\l}ynicki-{B}irula decomposition for reductive groups.
\newblock {\em Journal of Math\'ematiques Pures et Appliqu\'ees}, 131:290--325,
  2019.

\bibitem{KMM}
Y.~Kawamata, K.~Matsuda, and K.~Matsuki.
\newblock Introduction to the {M}inimal {M}odel {P}rogram.
\newblock {\em Adv. Stud. Pure Math}, 10:283--360, 1987.

\bibitem{KPSW}
S.~Kebekus, T.~Peternell, A.~J. Sommese, and J.~A. Wi{\'s}niewski.
\newblock Projective contact manifolds.
\newblock {\em Invent. Math.}, 142(1):1--15, 2000.

\bibitem{KNOP}
F.~Knop, H.~Kraft, D.~Luna, and T.~Vust.
\newblock Local properties of algebraic group actions.
\newblock In {\em Algebraische Transformationsgruppen und Invariantentheorie
  Algebraic Transformation Groups and Invariant Theory}, pages 63--75.
  Springer, 1989.

\bibitem{KOB}
S.~Kobayashi and T.~Ochiai.
\newblock Characterizations of complex projective spaces and hyperquadrics.
\newblock {\em J. Math. Kyoto Univ.}, 13(1):31--47, 1973.

\bibitem{KOLLAR_MORI}
J.~Koll{\'a}r and S.~Mori.
\newblock {\em Birational geometry of algebraic varieties}, volume 134.
\newblock Cambridge {U}niversity {P}ress, 2008.

\bibitem{LANGER}
A.~Langer.
\newblock {S}emistable sheaves in positive characteristic.
\newblock {\em Ann. of Math.}, 159 (1)(2):251--276, 2004.

\bibitem{L-S}
C.~LeBrun and S.~Salamon.
\newblock {Strong rigidity of positive quaternion-K\"ahler manifolds.}
\newblock {\em {Invent. Math.}}, 118(1):109--132, 1994.

\bibitem{MFK}
D.~Mumford, J.~Fogarty, and F.~Kirwan.
\newblock {\em Geometric invariant theory. {T}hird edition}, volume~34 of {\em
  Ergeb. Math. Grenzgeb.}
\newblock Springer-Verlag, Berlin, 1994.

\bibitem{OSCRW2}
G.~Occhetta, E.~A. Romano, L.~E.~Sol{\'a} Conde, and J.~A. Wi{\'s}niewski.
\newblock Small bandwidth varieties and birational geometry.
\newblock {\em arXiv preprint:1911.12129}, 2019.

\bibitem{OSCRW}
G.~Occhetta, E.~A. Romano, L.~E.~Sol{\'a} Conde, and J.~A. Wi{\'s}niewski.
\newblock High rank torus actions on contact manifolds.
\newblock {\em arXiv preprint:2004.005971}, 2020.

\bibitem{P-W}
T.~Peternell and J.~A. Wi\'{s}niewski.
\newblock On stability of tangent bundles of {F}ano manifolds with {$b_2=1$}.
\newblock {\em J. Algebraic Geom.}, 4(2):363--384, 1995.

\bibitem{QUILLEN}
D.~{Quillen}.
\newblock {The spectrum of an equivariant cohomology ring. I. II.}
\newblock {\em {Ann. Math. (2)}}, 94:549--572, 573--602, 1971.

\bibitem{SALAMON}
S.~Salamon.
\newblock Quaternionic {K}\"ahler manifolds.
\newblock {\em Invent. Math.}, 67(1):143--171, 1982.

\bibitem{SOM}
A.~J. Sommese.
\newblock Extension theorems for reductive group actions on compact {K}\"aehler
  manifolds.
\newblock {\em Math. Ann.}, 218(2):107--116, 1975.

\bibitem{WIS2}
J.~A. Wi{\'s}niewski.
\newblock On a conjecture of {M}ukai.
\newblock {\em {M}anuscripta {M}athematica}, 68(1):135--141, 1990.

\bibitem{WIS}
J.~A. Wi\'{s}niewski.
\newblock On contractions of extremal rays of {F}ano manifolds.
\newblock {\em J. Reine Angew. Math}, 417:141--157, 1991.

\bibitem{W_NOTES}
J.~A. Wi\'sniewski.
\newblock Toric {M}ori theory and {F}ano manifolds.
\newblock {\em Geometry of toric varieties. Lectures of the summer school,
  Grenoble, France, June 19--July 7, 2000}, pages 249--272, 2002.

\end{thebibliography}
\bibliographystyle{plain}

\end{document}